\documentclass[12pt,a4paper]{article}
\usepackage[utf8]{inputenc}
\usepackage[usenames,dvipsnames,svgnames,table]{xcolor}
\usepackage{amsthm,amsmath,amsfonts,
amssymb,graphicx,color,dsfont}
\usepackage[colorlinks=true,citecolor=blue,linkcolor=blue,unicode=true]{hyperref}
\usepackage[nottoc,numbib]{tocbibind}

\usepackage{comment}

\allowdisplaybreaks

\newtheorem{theorem}{Theorem}[section]
\newtheorem{proposition}[theorem]{Proposition}
\newtheorem{lemma}[theorem]{Lemma}
\newtheorem{cor}[theorem]{Corollary}

\theoremstyle{definition}
\newtheorem{rem}{Remark}
\newtheorem{ass}{Assumption}
\newtheorem{definition}[theorem]{Definition}

\renewcommand\a{\mathfrak{a}}
\renewcommand\b{\mathfrak{b}}
\renewcommand\c{\mathfrak{c}}

\renewcommand\u{{u}}
\renewcommand\d{\,d}
 \newcommand{\R}{\mathds{R}}
 \newcommand{\Z}{\mathds{Z}}
 
 \newcommand{\N}{\mathds{N}}
 \renewcommand{\P}{\mathds{P}}
 \newcommand{\E}{\mathds{E}}
\newcommand{\1}{\mathds{1}}

\newcommand\argmax{\operatornamewithlimits{argmax}}

\newcommand\Op{{\operatorname O_{\P}}}

\renewcommand\complement[1]{{#1^c}}

\newcommand{\MJ}{\color{blue}}

\title{Time-varying first-order autoregressive processes with irregular innovations}
\author{Hanna Gruber, Moritz Jirak}
\date{\today}

\begin{document}

\maketitle

\begin{abstract}
Both locally stationary processes and irregular models have had a long story of success in statistics and time series analysis. We combine both concepts and consider a time-varying, first-order autoregressive model with irregular innovations, where we assume that the coefficient function is H\"{o}lder continuous. To estimate this function, we use a quasi-maximum likelihood based approach. A precise control of this method demands a delicate analysis of extremes of certain weakly dependent processes, our main result being a concentration inequality for such quantities. Based on our analysis, upper and matching minimax lower bounds are derived, showing the optimality of our estimators. Unlike the regular case, the information theoretic complexity depends both on the smoothness and an additional shape parameter, characterizing the irregularity of the underlying distribution. The results and ideas for the proofs are very different from classical and more recent methods in connection with locally stationary processes.
\end{abstract}



\section{Introduction}\label{sec:introduction}

Consider a (first order) autoregressive process $X_k$, $k \in \Z$, formally given by
\begin{align}\label{intro:ar:1}
X_{k} = a X_{k-1} + \varepsilon_k, \quad k \in \Z.
\end{align}
A typical (theoretical) assumption in the literature is that the distribution of $\varepsilon_k$ is regular (e.g. Gaussian), that is, tools such as Cram\'{e}r-Rao efficiency and local asymptotic normality (LAN) are available to analyse estimators, see for instance \cite{durbin} or \cite{kreiss:aos:1987} in this context. However, in many situations, data display non regular features such as one sided support, and it is thus more appropriate to model the driving stochastic component $\varepsilon_k$ with an irregular distribution function, for instance a Gamma distribution. Such statistical models have found broad applications in dendroclimatology, hydrology, epidemiology, finance and quality control, see for instance \cite{bibinger-jirak-reiss2014}, \cite{diaz-hughes-swetnam2010}, \cite{delleur1984}, \cite{preve12030139} and \cite{smith_1994}. Particularly for an autoregressive context, we refer to \cite{BONDON:jmva:2009}, \cite{lawrance-lewis1980}, \cite{gaver-lewis1980}, \cite{PREVE2015S225}, \cite{barndorf:Jrssb:2001} and the books \cite{balakrishna:book:2021}, \cite{rosenblatt:book:2000}. In a non-parametric regression context, irregular models have also been more recently explored in \cite{hall_2009},\cite{jirak-meister-reiss2014},\cite{selk2021multivariate} and \cite{daouia_2021}. Finally, let us mention that irregular models are also of significant importance in the econometric literature, see \cite{farell1957}, \cite{aigner-lovell-schmidt1977}, \cite{meeusen-vandenbroeck1977} and \cite{park-sickles-simar1998}, \cite{park-sickles-simar2007}, \cite{kumbhakar-park-simar-tsionas2007} for some more recent accounts.

Turning more to mathematical aspects of irregular autoregressive processes, these have been studied, among others, in \cite{basawa:jtsa:2005}, \cite{davis-mccormick1989}, \cite{davis-knight-liu1992}, \cite{MR1288286}, \cite{feigin:1996:aaop}, \cite{nielsen-shephard1999}, \cite{knight2001}, \cite{ing:jasa:positive:auto:2012}, \cite{ing:jtsa:2018}. However, all these references give rise to parametric models and stationary processes, which is not always appropriate. Our aim is thus to make the next step and move to non-parametric, locally stationary autoregressive processes in the irregular realm. To be more precise, we replace $a$ in \eqref{intro:ar:1} with a function $f:[0,1]\mapsto[0,\rho]$ with $\rho\in(0,1)$, leading to
the time-varying first-order autoregressive (tvAR(1)) process given by
\begin{equation}\label{eq:AR-process}
X_k=f\left(\frac kN\right)X_{k-1}+\varepsilon_k,\quad k=1,\dots,N,
\end{equation}
with innovations $\varepsilon_k\geq 0$ (or $\varepsilon_k \leq 0$) and design points $k/N\in[0,1]$. Our aim here is to estimate $f$ subject to certain smoothness constraints, based on the observations $X_1,\ldots,X_N$. For convenience, we work with positive innovations and equidistant design points, but the results could be extended to negative innovations or inequidistant design points with analogous arguments. As it turns out, solving this task is not easy, and requires new ideas for the following reasons: First, previous techniques developed in the references above focus on weak convergence results, which is not sufficient for us, as we require (optimal) concentration results. Secondly, it is not at all clear how to adapt these methods to fit a non-parametric context, and we therefore develop our own approach.

There is, by now, a huge body of literature on various aspects of (regular) locally stationary processes, see for instance \cite{dahlhaus2012} for an overview, and \cite{dahlhaus_richter_wu_2019},\cite{richter_dahlhaus_2019},\cite{eichler_2011},\cite{kreiss_boot_2015}, \cite{Vogt_aos_2012}, for some more recent contributions. However, to the best of our knowledge, the past end present theory entirely focuses on regular models, and thus have an entirely different statistical complexity from an information theoretic perspective. As in classical, parametric statistics, irregular models are rather connected to order statistics and extreme value theory than to sample means, as is the case for regular models. In the present context, this leads to dealing with delicate problems involving extremes of weakly dependent processes. Another interesting feature of our irregular, time varying process \eqref{eq:AR-process} is that local, stationary approximations appear to be a pitfall in general. In the regular case, local stationarity is usually exploited by locally approximating the process by a stationary proxy in order to apply all the machinery available for stationary processes, see e.g. \cite{dahlhaus2012} for details. We dispense with such an approximation to avoid the resulting approximation error. In fact, this appears to be even necessary to attain the minimax rates in general, since otherwise the approximation error appears to be too large.

This work is structured as follows. In Section \ref{sec:methodology-main-results}, we present the basic setting and our main results. An outline of the proofs is given in Section \ref{sec:upper-bound} (upper bounds) and Section \ref{sec:lower-bound} (lower bounds). Detailed proofs are given in Section \ref{sec:proofs}, where some technical results are deferred to Supplement \ref{sec:supplement}.

\section{Methodology and main results}\label{sec:methodology-main-results}

We use the following notation. For sequences $a_n$, $b_n \in \R$ we write $a_n\sim b_n$ if $a_n\lesssim b_n$ and $a_n\gtrsim b_n$. Here, $a_n\lesssim b_n$ means that there exist $c\in(0,\infty)$, $n_0\in\N$, such that $a_n\le cb_n$ for all $n\ge n_0$, and $a_n\gtrsim b_n$ means that $b_n\lesssim a_n$. For a set $\mathcal{A}$, we denote with $|\mathcal{A}|$ its cardinality and $\mathcal{A}^c$ its complement. For a random variable $X$ and $p\ge1$, we denote with $\|\cdot\|_p$ the $L_p$-norm defined by $\|X\|_p^p = \E |X|^p$.  We write $\stackrel
{d}{=}$ for equality in distribution. Moreover, we use $a \wedge b = \min\{a,b\}$ and $a \vee b = \max\{a,b\}$.

Our key condition regarding the function $f$ is $f \in \mathcal{H}(L,\beta)$, that is, we assume $f$ to be a member of the Hölder class $\mathcal{H}(L,\beta)$ for $L$, $\beta>0$. This means that for the $\langle\beta\rangle$-derivatives of $f$, we have
\begin{equation}\label{eq:smoothness}
\left|f^{(\langle\beta\rangle)}(y)-f^{(\langle\beta\rangle)}(z)\right|\le L\left|y-z\right|^{\beta-\langle\beta\rangle}\hspace{5mm}\forall\;y,z\in[0,1],
\end{equation}
where $\langle\beta\rangle:=\max\{a\in\N_0:a<\beta\}$ is defined to be the largest integer less than $\beta$. Moreover, it will be convenient to extend the domain of $f$ to the whole real axis by setting
\begin{align}
f(u)=f(0) \quad \text{for all $u<0$, and} \quad f(u)=f(1) \quad \text{for all $u>1$}.
\end{align}

Regarding the irregularity of the model, we make the standard assumption that the distribution function $F_\varepsilon$ of the innovations decays with a certain sharpness, indicated by the parameter $\a$, at the endpoints of its domain. We thus assume, as $y \downarrow 0$, that
\begin{equation}\label{eq:sharpness-AR}
F_\varepsilon(y)=\c y^\a+ g(y), \quad \big|g(y)\big| \leq \c_g y^{\a+\delta},
\end{equation}
for some arbitrarily $\delta>0$ and constants $\a,\c,\c_{g},\in(0,\infty)$. Note that this implies $\P(\varepsilon \geq 0) = 1$.

Throughout this note, we work subject to the following assumptions.

\begin{ass}\label{ass:main}
There exist $L, \beta > 0$ and $\rho < 1$, such that $f \in \mathcal{H}(L,\beta)$ and $0 \leq f \leq \rho$. Moreover, $X_k$ follows the autoregressive model \eqref{eq:AR-process}, where
the innovations $\varepsilon_k$ are i.i.d. and satisfy:
\begin{itemize}
  \item[{\bf (i)}]  The local decay condition \eqref{eq:sharpness-AR} and $F_{\varepsilon} \in \mathcal{H}(L_{\varepsilon},\beta_{\varepsilon})$ for $L_{\varepsilon}, \beta_{\varepsilon} > 0$.
  \item[{\bf (ii)}] $\E \varepsilon_1^p <\infty$ for $p = \max\{1,2\a+\delta\}$, where $\delta > 0$ can be arbitrarily small.
\end{itemize}
\end{ass}

\begin{rem}
One may select the same value for $\delta$ in \eqref{eq:sharpness-AR} and Assumption \ref{ass:main}. We do so in the sequel,
and thus identify it as the same parameter.
\end{rem}

Given a fixed estimation point $x\in[0,1]$ and a bandwidth $h>0$, consider the local sample size
\begin{align}\label{defn:n}
n=n(h,x):=\left|\left\{k:\frac kN\in[x-h,x+h]\right\}\right|.
\end{align}
Thus, $n$ is proportional to $Nh$, that is $n\sim Nh$. For the bandwidth $h$ and fixed $x \in [0,1]$, our quasi-MLE $\hat f_h$ is defined as follows. Writing the approximation polynomials as
\begin{align*}
p(y)=\sum_{i=0}^{\beta + 1}b_i(y-x)^i,\quad y\in[0,1],
\end{align*}
we have to find the (local) optimal coefficients $\hat b_i$. To this end, we maximize the sum of $p$ evaluated at the design points $k/N$ within the band $[x-h,x+h]$,
\begin{equation*}
\left(\hat b_i\right)_{i=0}^{\beta}:=\argmax_{\left(b_i\right)_{i=0}^{\beta + 1}}
\sum_{\left|\frac kN-x\right|\le h}p\left(\frac kN\right),
\end{equation*}
subject to the constraints $p(k/N)\le Y_k$ for all $k$ with $|k/N-x|\le h$. Thus, we obtain a linear program, whose derivation can be motivated from the MLE if $\varepsilon_k \stackrel{d}{=} \varepsilon$ follows an exponential distribution. The estimator for $f(x)$ is given by the value at $x$ of the approximating polynomial, hence
\begin{equation*}
\hat f_h(x):=\hat b_0.
\end{equation*}

Our key result is the following concentration inequality for the estimator $\hat f_h$.

\begin{theorem}\label{thm:concentration}
Fix $x \in [0,1]$ and grant Assumption \ref{ass:main}. Then there exist finite, positive constants $c_1,c_2, c_3$ and $\nu > 0$, such that for all $0 \leq v \leq n^{\nu}$ and $h > 0$
\begin{equation*}
\sup_{f \in \mathcal{H}(\beta,L)} \P_f\Big(\big|\hat f_h(x)-f(x)\big| \geq c_1 h^{\beta} + c_2 n^{-1/\mathfrak{a}} v \Big) \leq c_3 e^{-v^{\mathfrak{a}}}.
\end{equation*}
Here, $n = n(h,x)$ is defined in \eqref{defn:n}.
\end{theorem}

The above theorem opens the door for a number of further interesting results. Setting
\begin{align}\label{eq:h:star}
h^{\ast} = N^{-\frac{1}{\mathfrak{a} \beta + 1}} \quad \text{and} \quad \hat{f} = \hat{f}_{h^{\ast}},
\end{align}
we immediately obtain the following corollary.

\begin{cor}\label{cor:main:1}
Fix $x \in [0,1]$ and grant Assumption \ref{ass:main}. Then
\begin{align*}
\big|\hat f(x)-f(x)\big|=\mathcal{O}_{\P_f}\left(N^{-\frac{\beta}{\a\beta+1}}\right).
\end{align*}
\end{cor}

Recall $\nu > 0$ in Theorem \ref{thm:concentration} and let $\tau_n \to \infty$ such that $\tau_n =o\big(n^{\nu}\big)$. We then consider the truncated estimator
\begin{align}
\hat{f}^{\tau_n}(x) = \hat{f}\1(|\hat{f}| \leq \tau_n) + \tau_n \1(|\hat{f}| > \tau_n).
\end{align}

\begin{cor}\label{cor:main:2}
Fix $x \in [0,1]$ and grant Assumption \ref{ass:main}. Then, for any $q \geq 1$, there exists $C > 0$ such that
\begin{align*}
\E_f \big|\hat{f}^{\tau_n}(x)-f(x)\big|^q \leq C N^{-\frac{q\beta}{\a\beta+1}}.
\end{align*}
\end{cor}

Having successfully established upper bounds, our next task is to find matching lower bounds. This is achieved by the following result.

\begin{theorem}\label{thm:main:lower-bound}
For $\mathfrak{a} \in (0,2)$, the minimax rate is $N^{-\frac\beta{\a\beta+1}}$. That is, there exists a distribution $F_{\varepsilon}$ satisfying Assumption \ref{ass:main}, such that for any $x \in [0,1]$, a lower bound on the pointwise error is given by
\begin{align*}
\lim_{N\to\infty}\inf_{\tilde f}\sup_{f\in\mathcal{H}(\beta,L)}\P_f\left(\left|\tilde f(x)-f(x)\right|>N^{-\frac\beta{\a\beta+1}}\right)>0,
\end{align*}
where the infimum is taken over all $\tilde{f}$, measurable with respect to $\sigma(X_1,\ldots,X_N)$.
\end{theorem}


From the viewpoint of irregular models, the case $\mathfrak{a} \geq 2$ turns out to be rather uninteresting, as it can be transferred to the regular setup. This is a bit surprising on first sight, since this is different from the regression case discussed for instance in \cite{jirak-meister-reiss2014}, \cite{hall_2009}.
Indeed, for $\a\ge2$, the process can simply be regularised by subtracting the mean, that is, by considering
\begin{align}\label{regularise}
X_k-\bar X_{N h^{\ast}}\approx f\left(X_{k-1}-\E X_{k-1}\right)+\varepsilon_k-\E \varepsilon_k,
\end{align}
leading to the regular minimax rate $N^{-\beta/(2\beta+1)}$. More precisely, the approximation error in \eqref{regularise} can be shown to be of magnitude $\sim N^{-\beta/(2\beta+1)}$ at the most, and one may then appeal to the regular theory mentioned above. Strictly speaking, one also has to adapt the proof for the lower bound in this case, we omit the details. Observe that one can always use \eqref{regularise}, but this leads to a suboptimal estimate for $\mathfrak{a} \in (0,2)$, even if $\E X_k$ is known.\\
\\
Having established an optimal concentration inequality as in Theorem \ref{thm:concentration} also opens the door for future research. For instance, one may think about the adaptive case where $\beta$ and $\mathfrak{a}$ are both unknown as in the i.i.d. regression case discussed in~\cite{jirak-meister-reiss2014}. Among other things, this requires estimation of $\mathfrak{a}$ based on the observations $X_k$, which is currently investigated.\\
\\
Let us now turn to the problem of prediction. As is well-known, (locally) autoregressive models are extremely useful for prediction of future values and a substantial theory has evolved around this matter, see for instance~\cite{brockwell_davis_2016} for classical results. In our setup, prediction also requires a careful handling of the bias, since, as briefly discussed above, the standard empirical mean estimator $\bar X_{N h^{\ast}}$ only achieves the rate $(N h^{\ast})^{-1/2}$, which is not sufficient for our cause. However, this problem can be circumvented by plugin-estimation. To this end, let

\begin{align*}
\hat{\varepsilon}_k &= X_k - \hat{f}^{\tau_n}\big(k/N\big)X_{k-1}, \quad 2 \leq k \leq N,\\
\hat{X}_{N+1} &= X_N \hat{f}^{\tau_n}\big(1\big) + \frac{1}{N-1}\sum_{k = 2}^N \hat{\varepsilon}_k.
\end{align*}

Our final result below establishes the desired optimality of our predictor $\hat{X}_{N+1}$.

\begin{theorem}\label{thm:main:prediction}
Grant Assumption \ref{ass:main}. Then there exists a finite constant $C > 0$, such that
\begin{align*}
\Big|\E_f \big|X_{N+1} - \hat{X}_{N+1} \big|^2 - \operatorname{Var}_f\big(\varepsilon_{N+1}\big) \Big| \leq C \Big(N^{-\frac{2\beta}{\a\beta+1}} \vee N^{-1} \Big).
\end{align*}
This bound is minimax optimal up to multiplicative constants.
\end{theorem}

The origin of the additional term $N^{-1}$ stems from the necessity to estimate the unknown mean $\E \varepsilon_k$.

\section{Outline for the upper bound}\label{sec:upper-bound}

{\bf From now on, we assume the validity of Assumption \ref{ass:main} without mentioning it any further.}\\
\\
Unfortunately, it seems to be impossible to apply the methods of \cite{nielsen-shephard1999}, \cite{knight2001}, \cite{davis-mccormick1989}, \cite{davis-knight-liu1992} and \cite{MR1288286} in our case, in particular, since we require concentration inequalities, and the latter all rely on weak convergence type arguments. Thus, our goal is to use a positive-error version of the quasi-maximum likelihood procedure developed in~\cite{jirak-meister-reiss2014} for nonparametric regression models, see Section \ref{subsec:preparation} for more details. To this end, we divide the defining equation \eqref{eq:AR-process} of the process by $X_{k-1}$. Thereby, we obtain
\begin{equation}\label{eq:AR-as-regression}
Y_k:=\frac{X_k}{X_{k-1}}=f\left(\frac kN\right)+\frac{\varepsilon_k}{X_{k-1}}=:f\left(\frac kN\right)+\tilde\varepsilon_k,
\end{equation}
the standard model of nonparametric regression with $f$ as the regression function and errors $\tilde\varepsilon_k$. Compared to~\cite{jirak-meister-reiss2014}, a huge difference constitutes the fact that the sequence $(\tilde\varepsilon_k)$ exhibits (weak) dependence, posing substantial challenges both for the upper and lower bound in the present context.

First, we need to ensure that an analogue of the sharpness condition \eqref{eq:sharpness-AR} holds for the modified innovations $\tilde\varepsilon_k$ from \eqref{eq:AR-as-regression}. This may be surprising on first sight, but the intuition here is that $X_k \gg 0$ with high probability due to its autoregressive structure, hence, the distribution $F_{\tilde\varepsilon_k}$ of $\tilde{\varepsilon}_k$ 'should' behave as $F_{\varepsilon}$.
\begin{proposition}\label{thm:err-distrib}
The distribution functions $F_{\tilde\varepsilon_k}$ of the modified innovations $\tilde\varepsilon_k$ satisfy
\begin{equation*}
F_{\tilde\varepsilon_k}(y)=\c_ky^\a+\mathcal{O}\left(y^{\a+\delta'}\right)\hspace{5mm}\text{as }y\downarrow0
\end{equation*}
for some $\delta'\in(0,\delta]$, where $c^{-1} \leq \c_k \leq c $ for some $c \in (0,\infty)$.
\end{proposition}

In the following, to simplify the notation, we drop the index $f$ for all probability measures $\P_f$ and corresponding expectations $\E_f$.

\subsection{Error decomposition}\label{subsec:preparation}

To determine the upper bound, the error decomposition into a deterministic and a stochastic part established in \cite{jirak-meister-reiss2014} is essential. It is given in Theorem 3.1 in \cite{jirak-meister-reiss2014} and states that for all $f\in\mathcal{H}(L,\beta)$ - adapted to our situation - we have
\begin{equation}\label{eq:error-decomposition}
\left|\hat f_h(x)-f(x)\right|\le c(\beta,L)h^{\beta}+c(\beta)\max_{j=1}^{2J(\beta)}\left\{Z_j(h,x):x+h\mathcal{I}_j\subseteq[0,1]\right\}
\end{equation}
for constants $c(\beta,L)$, $c(\beta)>0$ and $J(\beta)\in\N$ only depending on their respective arguments. Here,
\begin{equation*}
Z_j(h,x)=\min_{k=1}^N\left\{\tilde\varepsilon_k:\frac kN\in x+h\mathcal{I}_j\right\},
\end{equation*}
where
\begin{equation*}
\mathcal{I}_j=[-1+(j-1)/J(\beta),-1+j/J(\beta)].
\end{equation*}
Thus, each $Z_j(h,x)$ represents the minimum of the modified errors $\tilde\varepsilon_k$ on the respective bin $x+h\mathcal{I}_j$. Up to a factor $c(\beta)$, the stochastic part of the error boundary consists of the maximum of these bin minima $Z_j(h,x)$. In view of \eqref{eq:error-decomposition}, we aim to show
\begin{equation*}
\P\left(\max_{j=1}^{2J(\beta)}\left\{Z_j(h,x):x+h\mathcal{I}_j\subseteq[0,1]\right\}\ge  n^{-1/\a}v \right)\lesssim e^{-\c^{(1)}v^\a}
\end{equation*}
for some constant $\c^{(1)}>0$. As $J(\beta)<\infty$, it suffices to show
\begin{equation}\label{eq:concentration}
\P\left(Z_j(h,x)\ge  n^{-1/\a}v \right)\lesssim e^{-\c^{(1)}v^\a}, \quad j=1,\dots,2J(\beta),
\end{equation}
due to the union bound. Let
\begin{align}\label{n:j}
n_j = \big|\big\{k : \, k/N \in x+h\mathcal{I}_j \big\} \big|.
\end{align}

Then clearly $n_j \sim n$, and, by the above, our goal is thus to establish
\begin{equation*}
\P\left(\min_{k=1}^{n_j}\tilde\varepsilon_k\ge n^{-1/\a}v\right)\lesssim e^{-\c^{(1)}v^\a}
\end{equation*}
for $n\in\N$ large enough (we will also require a constraint on $v\in(0,\infty)$, recall Theorem \ref{thm:concentration}). To ease the notation, for
\begin{equation}\label{eq:u-shortcut}
u=u(n,v,\a)=n^{-1/\a}v,
\end{equation}
we shall mostly write $\u$ in what follows.

In order to employ a blocking argument from Leadbetter \cite{leadbetter1974},\cite{MR691492} (which will be explained in detail in subsection \ref{subsec:Leadbetter}), we write
\begin{equation}\label{eq:n-is-2Mcp(n)}
n=2Mn^\gamma
\end{equation}
for some $\gamma\in(0,1)$ and $M\in\N$, and also assume $n^\gamma$ to be an integer for simplicity.

We start by dividing $X_k$, $k\in\{1,\dots,n\}$, into two parts,
\begin{eqnarray}
X_k&=&\sum_{i=0}^\infty\left(\prod_{l=0}^{i-1}f\left(\frac{k-l}n\right)\right)\varepsilon_{k-i}=:\sum_{i=0}^\infty f_{k,i}\varepsilon_{k-i}\label{eq:f_(k,i)-Def} \nonumber \\
&=&\sum_{i=0}^{n^\gamma}f_{k,i}\varepsilon_{k-i}+\sum_{i>n^\gamma}f_{k,i}\varepsilon_{k-i} =:X_k^{(1)}+X_k^{(2)},\label{eq:X1X2}
\end{eqnarray}
where we recall that the process is defined on $\mathbb{Z}$. Due to the independence of the innovations $\varepsilon_k$, the $X_k^{(1)}$ are $n^\gamma$-dependent. For $X_k^{(2)}$, we have the following trivial result we repeatedly make use of and therefore state for the sake of reference.
\begin{lemma}\label{lem:Norm-X2}
For $X_k^{(2)}$ defined in \eqref{eq:X1X2}, we have (uniformly in $k$)
\begin{equation*}
\left\|X_k^{(2)}\right\|_1\lesssim\rho^{\left(n^\gamma\right)}
\end{equation*}
for $\rho < 1$ as in Assumption \ref{ass:main}. Moreover, we have (uniformly in $k$) $\left\|X_k\right\|_p < \infty$.
\end{lemma}


\subsection{Leadbetter's blocking argument}\label{subsec:Leadbetter}

As previously mentioned, in order to determine the rate of the stochastic part of the error boundary in \eqref{eq:error-decomposition}, we make use of a blocking argument introduced by Leadbetter \cite{leadbetter1974}. We cut the index set $\{1,\dots,n\}$ into $2M$ blocks of equal length $n^\gamma$. Since \eqref{eq:error-decomposition} only provides an upper bound, we are not interested in the exact distribution of the right-hand side. So we simply drop every other block, leading us to
\begin{eqnarray}
\mathcal{K}&:=&\{1,\dots,n^\gamma\}\cup\{2n^\gamma+1,\dots,3n^\gamma\}\cup\dots\cup\nonumber\\
&&\{(2M-2)n^\gamma+1,\dots,(2M-1)n^\gamma\}\nonumber\\
&=:&\mathcal{K}_1\cup\dots\cup\mathcal{K}_M\label{eq:blocking}
\end{eqnarray}
as a new index set. Handling the minimum on $\mathcal{K}$ suffices since $\mathcal{K}$ is a subset of $\{1,\dots,n\}$ and thus the minimum on $\mathcal{K}$ cannot be smaller than the minimum on $\{1,\dots,n\}$. The individual blocks $\mathcal{K}_m$ are separated from each other by $n^\gamma$ to exploit the $n^\gamma$-dependence of the $X_{k-1}^{(1)}$. This is in line by what Leadbetter dubbed Condition D$(u)$.
\begin{lemma}\label{prop:D}
\textbf{\fontshape{n}\selectfont{(Condition D$(u)$)}} For the minimum of the innovations $\tilde\varepsilon_k$ on the band $[x-h,x+h]$ holds
\begin{eqnarray*}
\P\left(\min_{k=1}^n\tilde\varepsilon_k\ge\u\right)&\le&\prod_{m=1}^M\P\left(\min_{k\in\mathcal{K}_m}
\frac{\varepsilon_k}{X_{k-1}^{(1)}}\ge\u\right)+R_1,
\end{eqnarray*}
where, for $\eta_n \in (0,\infty)$, we have the bound
\begin{align*}
R_1\lesssim\eta_n^{-1}n\rho^{\left(n^\gamma\right)}+ n \big(u \eta_n\big)^{\beta_{\varepsilon}}.
\end{align*}
\end{lemma}
Roughly speaking, Condition D$(u)$ means that we can bound the distribution of the minimum on the whole index set $\{1,\dots,n\}$ by a product of the minima on the blocks $\mathcal{K}_m$, $m=1,\dots,M$, where for the block minima, the $X_k^{(2)}$ part has been cut from $X_k$ to ensure $n^\gamma$-dependence between the blocks.\\
We still need to bound the minima on the blocks $\mathcal{K}_m$, $m\in\{1,\dots,M\}$. For this, we use the inclusion-exclusion principle, obtaining
\begin{eqnarray}
\P\left(\min_{k\in\mathcal{K}_m}\frac{\varepsilon_k}{X_{k-1}^{(1)}}\ge\u\right)
&\le&1-\sum_{k\in\mathcal{K}_m}\P\left(\varepsilon_k<\u X_{k-1}^{(1)}\right)+\nonumber\\
&&\sum_{\substack{k,l\in\mathcal{K}_m:\\k<l}}\P\left(\varepsilon_k<\u X_{k-1}^{(1)},\varepsilon_l<\u X_{l-1}^{(1)}\right),\label{eq:Aufteilung-Teilmengenminimum}
\end{eqnarray}
confer the proof of Proposition \ref{prop:Consequence-D'}. To be able to bound the last term in \eqref{eq:Aufteilung-Teilmengenminimum}, we need to verify Leadbetter's second condition, Condition D$'(u)$. For our (nonstationary) situation, Condition D$'(u)$ reads as follows.
\begin{lemma}\label{prop:D'}
\textbf{\fontshape{n}\selectfont{(Condition D$'(u)$)}} For all $m\in\{1,\dots,M\}$ holds
\begin{equation*}
\sum_{\substack{k,l\in\mathcal{K}_m:\\k<l}}\P\left(\varepsilon_k<\u X_{k-1}^{(1)},\varepsilon_l<\u X_{l-1}^{(1)}\right)\lesssim\left(\frac{v^\a}M\right)^2.
\end{equation*}
\end{lemma}
The following result is a consequence of Lemma \ref{prop:D'} and constitutes a key step towards proving our concentration inequality, which will be done in Subsection \ref{subsec:upper-bound-overall}.
\begin{proposition}\label{prop:Consequence-D'}
For all $m\in\{1,\dots,M\}$ holds
\begin{eqnarray*}
\P\left(\min_{k\in\mathcal{K}_m}\frac{\varepsilon_k}{X_{k-1}^{(1)}}\ge\u\right)&\le&1-\frac{\c^{(3)}v^\a}M+R_2,
\end{eqnarray*}
where $\c^{(3)}\in(0,\infty)$ and for $\eta_n \in (0,\infty)$
\begin{align*}
R_2\lesssim\frac{n^\frac{-\delta}{\a}v^{\a+\delta}+ n^2 \big(u \eta_n\big)^{\beta_{\varepsilon}} + \eta_n^{-1}n^2\rho^{\left(n^\gamma\right)}}M+\left(\frac{v^\a}M\right)^2.
\end{align*}
\end{proposition}

\subsection{Proof of Theorem \ref{thm:concentration} and Corollaries \ref{cor:main:1} and \ref{cor:main:2}}\label{subsec:upper-bound-overall}

From Lemma \ref{prop:D} and Proposition \ref{prop:Consequence-D'}, we can now derive an upper bound for the stochastic part in the error decomposition \eqref{eq:error-decomposition}.
\begin{theorem}\label{thm:main-stoch-result}
Suppose that
\begin{equation}\label{eq:finite-sample-assumption}
v = o\big(n^\frac{1}{1+\a}\big).
\end{equation}
Then the stochastic part of the error decomposition \eqref{eq:error-decomposition} satisfies
\begin{equation*}
\P\left(\max_{j=1}^{2J(\beta)}\left\{Z_j(h,x):x+h\mathcal{I}_j\subseteq[0,1]\right\}\ge n^{-1/\a}v\right)\lesssim e^{-\c^{(1)}v^\a},
\end{equation*}
where $\c^{(1)}\in(0,\infty)$.
\end{theorem}

Theorem \ref{thm:concentration} and Corollaries \ref{cor:main:1} and \ref{cor:main:2} are now simple consequences of Theorem \ref{thm:main-stoch-result}. For the sake of completeness, we explicitly state this below.

\begin{proof}[Proof of Theorem \ref{thm:concentration}]
Follows from Theorem \ref{thm:main-stoch-result} and the error decomposition \eqref{eq:error-decomposition}.
\end{proof}

\begin{proof}[Proof of Corollary \ref{cor:main:1}]
Is an immediate consequence of Theorem \ref{thm:concentration}.
\end{proof}

\begin{proof}[Proof of Corollary \ref{cor:main:2}]
Using the fact that for any $X \geq 0$, $q \geq 1$, we have
\begin{align*}
\E X^q = q\int_0^{\infty} x^{q-1}\P(X > x)dx,
\end{align*}
this follows from Theorem \ref{thm:concentration} and straightforward computations.
\end{proof}

\begin{proof}[Proof of Theorem \ref{thm:main:prediction}]
We first establish the upper bound. By Cauchy-Schwarz and Corollary \ref{cor:main:2}, we have
\begin{align*}
\E^2 \big|X_{k-1}\big(\hat{f}^{\tau_n}(k/N)-f(k/N)\big)\big|^2 &\leq \E \big|X_{k-1}\big|^4 \E \big|\hat{f}^{\tau_n}(k/N)-f(k/N)\big|^4 \\&\lesssim N^{-\frac{4\beta}{\a\beta+1}},
\end{align*}
which also implies
\begin{align*}
\E^2 \big|\varepsilon_k - \hat{\varepsilon}_k \big|^2 \lesssim N^{-\frac{4\beta}{\a\beta+1}}.
\end{align*}
Hence by the triangle inequality
\begin{align*}
\E^{1/2} \Big|\sum_{k = 1}^N (\varepsilon_k - \hat{\varepsilon}_k) \Big|^2 \lesssim \sqrt{N} N^{-\frac{\beta}{\a\beta+1}}.
\end{align*}
Piecing everything together and exploiting the fact that $\varepsilon_{N+1}$ is independent of $\varepsilon_N, \varepsilon_{N-1}, \ldots$, we obtain via the triangle inequality
\begin{align*}
\Big|\E \big|X_{N+1} - \hat{X}_{N+1} \big|^2 - \operatorname{Var}\big(\varepsilon_{N+1}\big) \Big| &\lesssim N^{-2}\operatorname{Var}\Big( \sum_{k = 2}^N \varepsilon_k \Big) +  N^{-\frac{2\beta}{\a\beta+1}} \\&\lesssim N^{-\frac{2\beta}{\a\beta+1}} \vee N^{-1},
\end{align*}
which completes the proof for the upper bound. Let us now turn to the lower bound. We assume first that $2\beta < \a \beta +1$.
It is well known that the optimal predictor is the conditional expectation $\E\big[X_{N+1}\big|X_{N},X_{N-1},\ldots \big] = X_{N} f\big(1\big) + \E \varepsilon_N$. Suppose now there exists $\tilde{X}_{N+1} \in \sigma(X_N,X_{N-1},\ldots)$, such that
\begin{align*}
\Big|\E \big|X_{N+1} - \tilde{X}_{N+1} \big|^2 - \operatorname{Var}\big(\varepsilon_{N+1}\big) \Big| \leq a_N , \quad a_N = o\big(N^{-\frac{2\beta}{\a\beta+1}}\big).
\end{align*}
Then, for any $c > 0$, we have $\P\big(\big|X_{N} f(1)- \tilde{X}_{N+1} \big| > c N^{-\frac{2\beta}{\a\beta+1}}\big) \to 0$ as $N$ increases. Since $\lim_{c'\to 0}\P\big(|X_N| > c'\big) = 1$, uniformly in $N$, this contradicts Theorem \ref{thm:main:lower-bound}, and thus establishes optimality if $2\beta < \a \beta +1$. It remains to treat case $2\beta \geq \a \beta +1$. To this end, we consider the special case where $\rho = 0$, $\a > 0$ is known, and $\varepsilon_k$ follows a Gamma distribution $\Gamma(\a,\b)$, where the rate $\b$ is unknown. Since in this case $\E \varepsilon_k = \a/\b$, minimax optimal prediction boils down to minimax optimal estimation of $\b$. However, standard arguments show that the minimax rate here is $N^{-1/2}$, which completes the proof.
\end{proof}

\section{Outline for the lower bound}\label{sec:lower-bound}

Due to the simultaneous irregularity and dependence of the underlying sequence, establishing the lower bound turned out to be surprisingly demanding.
We use the common Neyman-Pearson strategy, where it suffices to show that for some constant $c > 0$, we have
\begin{align}\label{eq:lower-bound:neyman:pearson:bound}
\sup_{x}\frac{x}{1+x}\P_1\Big(\frac{d\P_0}{d\P_1} \geq x \Big) \geq c
\end{align}
for appropriate hypothesis $H_0$, $H_1$ and corresponding probability measures $\P_1$, $\P_0$, see Theorem 2.1 in Tsybakov \cite{tsybakov2009} for details. Our hypothesis
will be constructed based on $f_0,f_1\in\mathcal{H}(\beta,L)$, satisfying
\begin{equation}\label{eq:lower-bound-separated-hypotheses}
\left|f_0-f_1\right|\gtrsim N^{-\beta/(\a\beta+1)}.
\end{equation}

We now prove that our rate $N^{-\beta/(\a\beta+1)}$ in Corollary \ref{cor:main:1} (and hence also Corollary \ref{cor:main:2}) is minimax optimal for $\a\in(0,2)$. By the above, we need to construct $H_0$, $H_1$. For our first hypothesis, we assume $f_0>0$, while we take $f_1=0$ for the second. From now on, we thus write $f:=f_0$ for simplicity and insert $0$ for $f_1$ where appropriate. This leads to the two models
\begin{align*}
H_0:X_k=f\big(k/N\big)X_{k-1}+\varepsilon_k\quad\text{ and }\quad H_1:X_k=\varepsilon_k.
\end{align*}
The $\varepsilon_k$'s distribution can be chosen in a specific manner, since we aim at bounding the minimax risk (or minimax error probability) from below. We take $\varepsilon_k\sim\Gamma(\a,\b)$ for some $\b\in(0,\infty)$, and thus
\begin{equation}\label{eq:lower-bound-eps-density}
f_\varepsilon(x)=\frac{\b^\a}{\Gamma(\a)}x^{\a-1}e^{-\b x}.
\end{equation}
Further, we write
\begin{equation}\label{def:lower-bound-hat-n}
n^{\ast} :=N h^{\ast} =NN^{-1/(\a\beta+1)}=N^{(\a\beta)/(\a\beta+1)}
\end{equation}
for the local sample size corresponding to the bandwidth $h^{\ast}$ in \eqref{eq:h:star}. In the following, $\P_j=\P_{f_j},j\in\{0,1\},$ denotes the joint distribution of $(X_1,\dots,X_{N})$ under $H_j$. To simplify the notation and calculations, we set $f(x) = f\mathds{1}_{\{x \leq n^{\ast}/N\}}$ for
\begin{equation}\label{eq:lower-bound-rate}
f=(c_f n^{\ast})^{-1/\a}, \quad c_f > 0.
\end{equation}
Clearly, $f(x)$ is not even continuous, but this (and $f \in \mathcal{H}(\beta,L)$) can be salvaged by a usual Kernel modification without effecting the rates. As pointed out above, we therefore stick to the current construction.

We start our lower bound proof by computing the Radon-Nikodym derivative.

\begin{proposition}\label{prop:lower-bound-Radon-Nikodym}
We have $\P_0\ll\P_1$. Moreover, the Radon-Nikodym derivative $d\P_0/d\P_1$ is given by
\begin{align*}
&\frac{d\P_0}{d\P_1}\left(x_1,\dots,x_{N}\right) = \frac{d\P_0}{d\P_1}\left(x_1,\dots,x_{n^{\ast}}\right)\\
=&e^{(\a-1)\sum_{k=1}^{n^{\ast}}\log\left(1-f\frac{x_{k-1}}{x_k}\mathds{1}_{\{x_k>fx_{k-1}\}}\right)+\b\sum_{k=1}^{n^{\ast}}fx_{k-1}}\prod_{k=1}^{n^{\ast}}\mathds{1}_{\{x_k>fx_{k-1}\}}
\end{align*}
for $x_i \in(0,\infty)$.
\end{proposition}

Next, we introduce a truncation, and obtain that for any $\tau > 0$, we have
\begin{align}\label{eq:lower-bound:truncation}
\P_1\Big(\frac{d\P_0}{d\P_1} \geq x \Big) \geq \P_1\Big(\frac{d\P_0}{d\P_1}\big(x_1,\dots,x_{n^{\ast}}\big)\mathds{1}_{\{\max_{0 \leq k \leq n^{\ast}-1}x_k \leq \tau\}} \geq x \Big).
\end{align}

Setting
\begin{align}
U_k&:=(\a-1)\log\left(1-f\frac{X_{k-1}}{X_k}\right)\mathds{1}_{\{X_k>fX_{k-1}, X_{k-1} \leq \tau\}}+\b fX_{k-1} \mathds{1}_{\{X_{k-1} \leq \tau\}},\label{eq:lower-bound-def-Uk}
\end{align}
the (truncated) Radon-Nikodym derivative can be written as
\begin{align}\label{eq:lower-bound-RD-derivative-random-variables}
\frac{d\P_0}{d\P_1}\left(X_1,\dots,X_{n^{\ast}}\right)\mathds{1}_{\max_{0 \leq k \leq n^{\ast}-1} X_k \leq \tau} &=e^{\sum_{k=1}^{n^{\ast}}U_k}\prod_{k=1}^{n^{\ast}}\mathds{1}_{\{X_k>fX_{k-1}, X_{k-1} \leq \tau\}}.
\end{align}

The following result establishes the correct order of the first two moments of $U_k$, which is one of the key results.

\begin{proposition}\label{prop:lower-bound-moments-U-1}
Let $2f \tau \leq 1$. Then for $U_k$ as defined in \eqref{eq:lower-bound-def-Uk}, we have uniformly for $1 \leq k \leq n^{\ast}$:
\begin{align*}
&{\bf (i)} \quad |\E_{\P_1} U_k |\lesssim f^\a,\\
&{\bf (ii)} \quad \E_{\P_1} U_k^2 \lesssim f^{\a}.
\end{align*}
\end{proposition}

Using the above proposition and some further estimates, we are now in position to establish the lower bound, that is, Theorem \ref{thm:main:lower-bound}.

\section{Proofs of the main results}\label{sec:proofs}

\subsection{Proofs for the upper bounds}

In the following, to simplify the notation, we drop index $f$ for all probability measures $\P_f$ and corresponding expectations $\E_f$.

\begin{proof}[Proof of Proposition \ref{thm:err-distrib}]
Fix an arbitrary $k\in\{1,\dots,N\}$. We have to show that for any sequence $y_N$ with $y_N\downarrow0$,
$$F_{\tilde\varepsilon_k}(y_N)=\c_ky_N^\a+\mathcal{O}\left(y_N^{\a+\delta'}\right)$$
holds. With $\varepsilon_k$ being independent of $X_{k-1}$, the distribution function $F_{\tilde\varepsilon_k}$ of $\tilde\varepsilon_k$ can be written as
\begin{eqnarray*}
F_{\tilde\varepsilon_k}(y_N)&=&\P(\varepsilon_k\le X_{k-1}y_N)\\
&=&\int_0^\infty\P(\varepsilon_k\le xy_N|X_{k-1}=x)\;dF_{X_{k-1}}(x)\\
&=&\int_0^\infty\P(\varepsilon_k\le xy_N)\;dF_{X_{k-1}}(x).
\end{eqnarray*}
For any given $y_N\downarrow0$ we can choose a sequence $x_N\to\infty$ such that $x_Ny_N\downarrow0$ as $N\to\infty$. More precisely, set $x_N:=y_N^{\chi-1}$ for some $\chi\in(0,1)$. Employing \eqref{eq:sharpness-AR} then yields
\begin{eqnarray*}
\int_0^\infty\P\left(\varepsilon_k\le xy_N\right)\;dF_{X_{k-1}}(x)
&=&\int_0^{x_N}\!\!\c(xy_N)^\a+\mathcal{O}\left((xy_N)^{\a+\delta}\right)\;dF_{X_{k-1}}(x)\\
&&+\int_{x_N}^\infty\P(\varepsilon_k\le xy_N)\;dF_{X_{k-1}}(x)\\
&=:&I_N+II_N.
\end{eqnarray*}
With Assumption \ref{ass:main} and Lemma \ref{lem:Norm-X2} guaranteeing (uniformly in $k$)
\begin{align*}
\E X_{k-1}^{\a+t}-\E X_{k-1}^{\a+t}\mathds{1}_{\{X_{k-1}>x_N\}} \le\E X_{k-1}^{\a+t} <\infty,
\end{align*}
for $t\in[0,\delta]$ we obtain
\begin{eqnarray}\label{eq:I_n-addends}
I_N&=&\c y_N^\a\left(\E X_{k-1}^\a-\E X_{k-1}^\a\mathds{1}_{\{X_{k-1}>x_N\}} \right)+\nonumber\\
&&\mathcal{O}\left(y_N^{\a+\delta}\right)\left(\E X_{k-1}^{\a+\delta}-\E X_{k-1}^{\a+\delta}\mathds{1}_{\{X_{k-1}>x_N\}} \right)\nonumber\\
&=:&\c_ky_N^\a+\mathcal{O}\left(y_N^{\a+\delta}\right).
\end{eqnarray}
Since for large enough $x_N$ (uniformly in $k$)
\begin{align*}
\E X_{k-1}^\a-\E X_{k-1}^\a\mathds{1}_{\{X_{k-1}>x_N\}} \geq \frac{1}{2} \E X_{k-1}^{\a} \geq \frac{1}{2} \E \varepsilon_{k-1}^{\a},
\end{align*}
and $\E X_{k-1}^{\a} \lesssim 1$ by Lemma \ref{lem:Norm-X2}, we arrive at $c^{-1} \leq \c_k \leq c$ for $c \in (0,\infty)$.

We now proceed by showing $II_N=\mathcal{O}(y_N^{\a+\delta_1})$. Since
\begin{eqnarray*}
II_N&=&\int_{x_N}^\infty\P(\varepsilon_k\le xy_N)\;dF_{X_{k-1}}(x)
\le\int_{x_N}^\infty dF_{X_{k-1}}(x)
=\P(X_{k-1}>x_N),
\end{eqnarray*}
we can apply Markov's inequality, yielding
$$\P(X_{k-1}>x_N)\le\|X_{k-1}\|_p^p x_N^{-p}.$$
Using Assumption \ref{ass:main} {\bf (ii)} and selecting $\chi > 0$ sufficiently small, this implies
\begin{align*}
\|X_{k-1}\|_{p}^{p}x_N^{-p}=\|X_{k-1}\|_{p}^{p}y_N^{-p(\chi-1)}\lesssim y_N^{p(1-\chi)}\lesssim y_N^{\a+\delta_1}
\end{align*}
for some $\delta_1 > 0$, yielding $II_N \lesssim y_N^{\a+\delta_1}$. As $y_N\downarrow0$ was arbitrary, setting $\delta':=\min\{\delta,\delta_1\}$ results in
\begin{align*}
F_{\tilde\varepsilon_k}(y)=\c_ky^\a+\mathcal{O}\left(y^{\a+\delta'}\right)\hspace{3mm}\text{ as }y\downarrow0.
\end{align*}
\end{proof}

\begin{proof}[Proof of Lemma \ref{prop:D}]
Let
$$C_n:=\left\{\min_{k=1}^n\tilde\varepsilon_k\ge\u\right\}=\bigcap_{k=1}^n\{\varepsilon_k\ge\u X_{k-1}\}$$
be the set whose probability we want to estimate, and
\begin{equation}\label{eq:definition-A_eta_n}
A_{\eta_n}:=\left\{\max_{k=1}^nX_{k-1}^{(2)}<\eta_n\right\}
\end{equation}
for some sequence $\left(\eta_n\right)\downarrow0$. Further, denote with $C_{\eta_n,n} = A_{\eta_n} \cap C_n$,
\begin{equation*}
C_n^{(1)}:=\bigcap_{k=1}^n\{\varepsilon_k\ge\u X_{k-1}^{(1)}\}
\end{equation*}
and $C_{\eta_n,n}^{(1)}=C_n^{(1)}\cap A_{\eta_n}$. Using Markov's inequality and Lemma \ref{lem:Norm-X2}, we obtain
\begin{align}\label{eq:eta:bound}
\P(\complement{A_{\eta_n}})
&\le\sum_{k=1}^n\P\left(X_{k-1}^{(2)}\ge\eta_n\right)\le\eta_n^{-1}\sum_{k=1}^n\left\|X_{k-1}^{(2)}\right\|_1
\lesssim\eta_n^{-1}n\rho^{\left(n^\gamma\right)}.
\end{align}

As both $C_n\setminus A_{\eta_n}$ and $C_n^{(1)}\setminus A_{\eta_n}$ are subsets of $\complement{A_{\eta_n}}$, we conclude from \eqref{eq:eta:bound} that
\begin{equation}\label{eq:Leadbetters-D-b}
|\P(C_n)-\P(C_{\eta_n,n})|=\P(C_n\setminus A_{\eta_n})\lesssim\eta_n^{-1}n\rho^{\left(n^\gamma\right)},
\end{equation}
as well as
\begin{eqnarray}
\left|\P\left(C_{\eta_n,n}^{(1)}\right)-\P\left(C_n^{(1)}\right)\right|
&=&\P\left(C_n^{(1)}\backslash A_{\eta_n}\right)
\lesssim\eta_n^{-1}n\rho^{\left(n^\gamma\right)}.\label{eq:Leadbetters-D-d}
\end{eqnarray}
Next, we consider $|\P(C_{\eta_n})-\P(C_{\eta_n,n}^{(1)})|$ and $\P(C_n^{(1)})$. We require the following lemma, whose proof can be found in the Supplement.
\begin{lemma}\label{lem:Leadbetters-D-c}
We have
\begin{equation*}
\left|\P\left(C_{\eta_n,n}\right)-\P\left(C_{\eta_n,n}^{(1)}\right)\right|\lesssim n  \big(u \eta_n \big)^{\beta_{\varepsilon}}.
\end{equation*}
\end{lemma}

We can now use the above results to complete the proof. Combining \eqref{eq:Leadbetters-D-b}, \eqref{eq:Leadbetters-D-d}, and Lemma \ref{lem:Leadbetters-D-c}, we have
\begin{eqnarray*}
\P\left(C_n\right)&\le&\left|\P\left(C_n\right)-\P\left(C_{\eta_n,n}\right)\right|+\left|\P\left(C_{\eta_n,n}\right)-
\P\left(C_{\eta_n,n}^{(1)}\right)\right|+\\
&&\left|\P\left(C_{\eta_n,n}^{(1)}\right)-\P\left(C_n^{(1)}\right)\right|+\P\left(C_n^{(1)}\right)\\
&\le&\P\left(C_n^{(1)}\right)+R_1,
\end{eqnarray*}
with $R_1$ as asserted. Using the blocking argument from \eqref{eq:blocking} and the $n^\gamma$-dependence of the $X_{k-1}^{(1)}$ gives us
\begin{align*}
\P\left(C_n^{(1)}\right)
\le\P\left(\bigcap_{m=1}^M\bigcap_{k\in\mathcal{K}_m}\left\{\varepsilon_k\ge\u X_{k-1}^{(1)}\right\}\right)
=\prod_{m=1}^M\P\left(\min_{k\in\mathcal{K}_m}\frac{\varepsilon_k}{X_{k-1}^{(1)}}\ge\u\right),
\end{align*}
and hence
\begin{eqnarray*}
\P\left(\min_{k=1}^n\tilde\varepsilon_k\ge\u\right)&
\le&\prod_{m=1}^M\P\left(\min_{k\in\mathcal{K}_m}\frac{\varepsilon_k}{X_{k-1}^{(1)}}\ge\u\right)+R_1.
\end{eqnarray*}
\end{proof}

\begin{proof}[Proof of Lemma \ref{prop:D'}]
We use the following key Lemma, whose proof is given in the Supplement.
\begin{lemma}\label{lem:neighboring}
For all $j\ge1$, we have
\begin{equation*}
\P\left(\varepsilon_k\le\u X_{k-1},\varepsilon_{k+j}\le\u X_{k+j-1}\right)\lesssim\u^{2\a}.
\end{equation*}
\end{lemma}
Due to $X_{k-1}^{(1)}\le X_{k-1}$ for all $k$, we have
\begin{equation*}
\left\{\varepsilon_k\le uX_{k-1}^{(1)}\right\}\subseteq\{\varepsilon_k\le uX_{k-1}\},
\end{equation*}
and therefore also
\begin{equation*}
\left\{\varepsilon_k\le uX_{k-1}^{(1)},\varepsilon_l\le uX_{l-1}^{(1)}\right\}\subseteq\left\{\varepsilon_k\le
uX_{k-1},\varepsilon_l\le uX_{l-1}\right\}
\end{equation*}
for all $k$, $l\in\{1,\dots,n\}$. By this and Lemma \ref{lem:neighboring}, we conclude
\begin{eqnarray*}
\sum_{\substack{k,l\in\mathcal{K}_m:\\k<l}}\P\left(\varepsilon_k\le\u X_{k-1}^{(1)},\varepsilon_l\le\u X_{l-1}^{(1)}\right)
&\le&\sum_{\substack{k,l\in\mathcal{K}_m:\\k<l}}\P\left(\varepsilon_k\le\u X_{k-1},\varepsilon_l\le\u X_{l-1}\right)\\
&\lesssim&\left|\left\{(k,l)\in\mathcal{K}_m^2:k<l\right\}\right|\u^{2\a}
\lesssim\left(n^\gamma\u^\a\right)^2\\
&=&\left(\frac{v^\a}M\right)^2,
\end{eqnarray*}
where we also have used \eqref{eq:u-shortcut} and \eqref{eq:n-is-2Mcp(n)}.
\end{proof}

\begin{proof}[Proof of Proposition \ref{prop:Consequence-D'}]
From the inclusion-exclusion principle as described in \eqref{eq:Aufteilung-Teilmengenminimum}, using also Lemma \ref{prop:D'}, we get
\begin{align*}
&\P\left(\min_{k\in\mathcal{K}_m}\frac{\varepsilon_k}{X_{k-1}^{(1)}}\ge\u\right)\\
\le&1-\sum_{k\in\mathcal{K}_m}\P\left(\varepsilon_k\le\u X_{k-1}^{(1)}\right)+\sum_{\substack{k,l\in\mathcal{K}_m:\\k<l}}\P\left(\varepsilon_k\le\u X_{k-1}^{(1)},\varepsilon_l\le\u X_{l-1}^{(1)}\right)\\
\le&1-\sum_{k\in\mathcal{K}_m}\P\left(\varepsilon_k\le\u X_{k-1}\right)+\sum_{k\in\mathcal{K}_m}\P\left(\u X_{k-1}^{(1)}<\varepsilon_k\le\u X_{k-1}^{(1)}+\u X_{k-1}^{(2)}\right)
+R_3,
\end{align*}
where $R_3\lesssim(v^\a/M)^2$. Now we again make use of the set $A_{\eta_n}$ from \eqref{eq:definition-A_eta_n}, and use the estimate \eqref{eq:eta:bound} and the arguments following \eqref{eq:for-each-summand-for-later} in the proof of Lemma \ref{lem:Leadbetters-D-c} once more. Thereby, and with the sharpness condition on $\tilde\varepsilon_k$ from Proposition \ref{thm:err-distrib}, we establish that $\P(\min_{k\in\mathcal{K}_m}\varepsilon_k/X_{k-1}^{(1)}\ge\u)$ can further be bounded by
\begin{eqnarray*}
&&1-\sum_{k\in\mathcal{K}_m}\left(\c_k\u^\a+\mathcal{O}\left(\u^{\a+\delta'}\right)\right)+\\
&&\sum_{k\in\mathcal{K}_m}\P\left(\left\{\u X_{k-1}^{(1)}<\varepsilon_k\le\u X_{k-1}^{(1)}+\u X_{k-1}^{(2)} \right\}\cap A_{\eta_n}\right)+\\
&&\sum_{k\in\mathcal{K}_m}\P\left(\left\{\u X_{k-1}^{(1)}<\varepsilon_k\le\u X_{k-1}^{(1)}+\u X_{k-1}^{(2)} \right\}\cap\complement{A_{\eta_n}}\right)+R_3\\
&\le&1-\c'\u^\a\sum_{k\in\mathcal{K}_m} + n^\gamma\mathcal{O}\left(\u^{\a+\delta'}\right)+R_4,
\end{eqnarray*}
where $\c' > 0$ and
\begin{align*}
R_4&\lesssim n^\gamma  n(u \eta_n)^{\beta_{\varepsilon}} + n^\gamma\eta_n^{-1}n\rho^{\left(n^\gamma\right)}+R_3\\
&\lesssim\frac{n^2(u \eta_n)^{\beta_{\varepsilon}}+\eta_n^{-1}n^2\rho^{\left(n^\gamma\right)}}{M}+\left(\frac{v^\a}M\right)^2.
\end{align*}
Recalling \eqref{eq:u-shortcut} and \eqref{eq:n-is-2Mcp(n)}, we thus conclude
\begin{eqnarray*}
\P\left(\min_{k\in\mathcal{K}_m}\frac{\varepsilon_k}{X_{k-1}^{(1)}}\ge\u\right)&\le&1-\c'\u^\a \E \varepsilon_1^\a n^\gamma+n^\gamma\mathcal{O}\left(\u^{\a+\delta}\right)+R_4\\
&=&1-\frac{\c' \E\varepsilon_1^\a v^\a}{2M}+\mathcal{O}\left(\frac{n^\frac{-\delta}{\a}v^{\a+\delta}}M\right)+R_4\\
&=&1-\frac{\c^{(3)}v^\a}M+R_2,
\end{eqnarray*}
where $\c^{(3)}=\c\E[\varepsilon_1^\a]/2$ and $R_2$ as asserted.
\end{proof}

\begin{proof}[Proof of Theorem \ref{thm:main-stoch-result}]
We can assume w.l.o.g. that $v \geq 1$. Combining Lemma \ref{prop:D} and Proposition \ref{prop:Consequence-D'}, using also \eqref{eq:u-shortcut}, we obtain
\begin{eqnarray}
\P\left(\min_{k=1}^n\tilde\varepsilon_k\ge\u\right)&\le&\Bigg(1-\frac{\c^{(3)}v^\a}M+R_2\Bigg)^M+R_1,\label{eq:MinAbschaetzungExakt}
\end{eqnarray}
where
\begin{align*}
R_1\lesssim\eta_n^{-1}n\rho^{\left(n^\gamma\right)}+ n \big(u \eta_n\big)^{\beta_{\varepsilon}},
\end{align*}
and
\begin{align*}
R_2\lesssim\frac{\left(n^{-\frac\delta\a}v^\delta+ n^2 \big(u \eta_n\big)^{\beta_{\varepsilon}} + \eta_n^{-1}n^2\rho^{\left(n^\gamma\right)}+\frac{v^\a}M\right)v^\a}M.
\end{align*}
Set $\eta_n = \rho^{\frac{\beta_{\varepsilon} n^{\gamma}}{1 + \beta_{\varepsilon}}}$ and observe
\begin{align*}
\max_{\gamma \in (0,1)} \frac{1- \gamma}{\a} \wedge \gamma = \frac{1}{1 + \a}.
\end{align*}
Using \eqref{eq:finite-sample-assumption}, we thus obtain
\begin{eqnarray}
R_1&\lesssim  \exp(-c'v^\a),
\label{eq:hintere2SummandenRHS}
\end{eqnarray}
$c' > 0$. Next, again \eqref{eq:finite-sample-assumption} and the above implies
\begin{eqnarray*}
n^{-\frac\delta\a}v^\delta+\left(1+u^\delta\right)\eta_n^{\frac\a{1+\a}}+v^{-\a}\eta_n^{-1}n^2\rho^{\left(n^\gamma\right)}+\frac{v^\a}M \to 0
\end{eqnarray*}
as $n\to\infty$, hence $R_2 = o\big(v^{\a}/N\big)$. Piecing everything together yields
\begin{align*}
\P\left(\min_{k=1}^n\tilde\varepsilon_k\ge\u\right)
\lesssim\left(1-\frac{(1-c'')\c^{(3)}v^\a}M\right)^M+e^{-\c^{(1)}v^\a}
\lesssim e^{-\c^{(1)}v^\a}
\end{align*}
for arbitrarily small $c'' > 0$. As we have argued in Subsection \ref{subsec:preparation}, the rate of each bin minimum $Z_j(h,x)$ on the bin $x+h\mathcal{I}_j$ has to be the same as the rate of the band minimum $\min_{k=1}^{n_j} \tilde\varepsilon_k$. Therefore, for the maximum of the bin minima equally holds
\begin{align}
\P\left(\max_{j=1}^{2J(\beta)}\{Z_j(h,x):x+h\mathcal{I}_j\subseteq[0,1]\}\ge u\right)\,\,\lesssim\,\,e^{-\c^{(1)}v^\a},
\end{align}
which completes the proof.
\end{proof}

\subsection{Proofs for the lower bound}

\begin{proof}[Proof of Proposition \ref{prop:lower-bound-Radon-Nikodym}]
Consider the AR(1)-Process $Y_k = a Y_{k-1} + \varepsilon_k$. Straightforward computations then reveal
\begin{align*}
\frac{d \P_1}{d \lambda}\Big(y_1,\ldots,y_N\Big) = f_{Y_1}(y_1)\prod_{i = 2}^N f_{\varepsilon}\big(y_i - a y_{k-1}\big),
\end{align*}
where $\lambda$ denotes the Lebesgue-measure. It is now easy to see that $\P_0 \ll \P_1$, and, since $\P_0$ and $\P_1$ only differ on the first $n^{\ast}$-coordinates, it follows that
\begin{align*}
&\frac{d\P_0}{d\P_1}\left(x_1,\dots,x_{N}\right)\\
=&\prod_{k=1}^{n^{\ast}}\left(1-f\frac{x_{k-1}}{x_k}\right)^{\a-1}e^{\b f x_{k-1}}\mathds{1}_{\{x_k>f x_{k-1}\}}\\
=&e^{(\a-1)\sum_{k=1}^{n^{\ast}}\log\left(1-f\frac{x_{k-1}}{x_k}\mathds{1}_{\{x_k>fx_{k-1}\}}\right)+\b \sum_{k=1}^{ n^{\ast}}fx_{k-1}}\prod_{k=1}^{n^{\ast}}\mathds{1}_{\{x_k>fx_{k-1}\}},
\end{align*}
which completes the proof. The latter computation may also be used to establish $\P_0 \ll \P_1$.
\end{proof}

\begin{proof}[Proof of Proposition \ref{prop:lower-bound-moments-U-1}]
We need the following two lemmas whose proofs are given in the Supplement.
\begin{lemma}\label{lem:lower-bound-moments-U-2}
For $\a\geq 1$, we have for any $\tau > 0$
\begin{align*}
\Big|(\a-1)f\E_{\P_1}\Big[\frac{X_{k-1}}{X_k}\mathds{1}_{\{X_k>fX_{k-1}\}}\mathds{1}_{\{X_{k-1} \leq \tau\}}\Big] - \b f\E_{\P_1}\big[ X_{k-1}\mathds{1}_{\{X_{k-1} \leq \tau\}}\big]\Big| \lesssim f^{\a}.
\end{align*}
\end{lemma}

\begin{lemma}\label{lem:lower-bound-moments-U-1}
Denoting by $\|\cdot\|_p$ the $p$-norm with respect to $\P_1$, for $p>\a$, we have for any $\tau > 0$
\begin{align*}
\left\|f\frac{X_{k-1}}{X_k}\mathds{1}_{\{X_k>fX_{k-1}\}} \mathds{1}_{\{X_{k-1}\leq \tau\}} \right\|_p^p&\lesssim\frac{f^\a}{p-\a}+ f^p \tau^{(p-\mathfrak{a})\vee 0}\quad{\text{ as }}f\to0.
\end{align*}
\end{lemma}

Expanding $\log(1 - x)$, we have
\begin{eqnarray*}
\left|\E_{\P_1} U_k \right|& \leq & |1-\a|\sum_{p\ge2}\frac{1}{p}\E_{\P_1}\left[\left(f\frac{X_{k-1}}{X_k}\right)^p\mathds{1}_{\{X_k>fX_{k-1},X_{k-1} \leq \tau\}}\right]+\\
&&\Big| (\a-1)f\E_{\P_1}\left[\frac{X_{k-1}}{X_k}\mathds{1}_{\{X_k>fX_{k-1},X_{k-1} \leq \tau\}}\right] - \b f\E_{\P_1}\left[X_{k-1}\mathds{1}_{\{X_{k-1} \leq \tau\}}\right]\Big|.
\end{eqnarray*}
An application of Lemma \ref{lem:lower-bound-moments-U-2} ($\a \geq 1$) or Lemma \ref{lem:lower-bound-moments-U-1} ($\a < 1$) hence yields
\begin{align*}
\big|\E_{\P_1}[U_k]\big|&\lesssim \left|\sum_{p\ge2}\frac{1}{p}\E_{\P_1}\left[\left(f\frac{X_{k-1}}{X_k}\right)^p\mathds{1}_{\{X_k>fX_{k-1},X_{k-1} \leq \tau\}}\right]\right|+f^\a.
\end{align*}
Due to $p\ge2$, we have $p>\a$, and can thus employ Lemma \ref{lem:lower-bound-moments-U-1}, yielding
\begin{align*}
\sum_{p\ge2}\frac1p\left\|f\frac{X_{k-1}}{X_k}\mathds{1}_{\{X_k>fX_{k-1},X_{k-1} \leq \tau\}}\right\|_p^p
&\lesssim f^\a\left(\sum_{p\ge2}\frac{1}{p^2-\a p}+\sum_{p\ge2}\frac{(f \tau)^{p-\a}}{p}\right)
\lesssim f^\a.
\end{align*}
The last step is true since both series here converge (recall $2f \tau \leq 1$), and hence {\bf (i)} follows. For {\bf (ii)}, we can argue in a similar manner.
Noting $|\a-1|\le1$ and $f\lesssim f^{\a/2}$, we again obtain from expanding $\log(1 - x)$ and the triangle inequality
\begin{align*}
\left\|U_k\right\|_2
&\le\left|\a-1\right|\left\|\log\left(1-f\frac{X_{k-1}}{X_k}\right)\mathds{1}_{\{X_k>fX_{k-1},X_{k-1} \leq \tau\}}\right\|_2+\b f\left\|X_{k-1}\right\|_2\\
&\lesssim\sum_{p=1}^\infty\frac1p\left(\left\|f\frac{X_{k-1}}{X_k}\mathds{1}_{\{X_k>fX_{k-1},X_{k-1} \leq \tau\}}\right\|_{2p}^{2p}\right)^{1/2}+f^{\a/2}.
\end{align*}
Due to $2p >\a$, we can apply Lemma \ref{lem:lower-bound-moments-U-1}, yielding
\begin{align*}
\sum_{p=1}^\infty\frac1p\left(\left\|f\frac{X_{k-1}}{X_k}\mathds{1}_{\{X_k>fX_{k-1},X_{k-1} \leq \tau\}}\right\|_{2p}^{2p}\right)^{1/2}
&\lesssim f^{\a/2}\sum_{p=1}^\infty\left(\frac{1}{2p^3-\a p^2}+\frac{(f \tau)^{p-\a}}{p^2}\right)^{1/2}
\lesssim f^{\a/2},
\end{align*}
which completes the proof.
\end{proof}

\begin{proof}[Proof of Theorem \ref{thm:main:lower-bound}]
We require the following two additional results.
\begin{lemma}\label{lem:lower-bound-product-of-Einsen}
We have
\begin{equation*}
\prod_{k = 1}^{n^{\ast}} \mathds{1}_{\{X_k > f X_{k-1}\}} \xrightarrow{\P_1} 1 \quad\text{ as }c_f\to\infty.
\end{equation*}
\end{lemma}
\begin{lemma}\label{lem:control:truncation}
Let $\tau = \log^2 N$. Then
\begin{align*}
\P_1\big(\max_{0 \leq k \leq N} X_k \leq \tau \big) \to 1.
\end{align*}
\end{lemma}
Both Lemma \ref{lem:lower-bound-product-of-Einsen} and Lemma \ref{lem:control:truncation} are proven in the Supplement. In the sequel, it will also be convenient to use
\begin{equation}\label{eq:lower-bound-Vk}
V_k=U_k-\E_{\P_1} U_k, \quad 1 \leq k \leq n^{\ast}.
\end{equation}
Observe that by Proposition \ref{prop:lower-bound-moments-U-1}, we have
\begin{align}\label{eq:lower_bound:Vk:second_moment}
\big\|V_k\big\|_2 \lesssim f^{\a/2} + f^{\a} \lesssim f^{\a/2}.
\end{align}
Let $\tau = \log^2 N$. Due to \eqref{eq:lower-bound:truncation} and \eqref{eq:lower-bound-RD-derivative-random-variables}, it suffices to show
\begin{align*}
&{\bf (i):}\,\, \sum_{k = 1}^{n^{\ast}} U_k \xrightarrow{\P_1} 0,\\
&{\bf (ii):}\,\, \prod_{k=1}^{n^{\ast}}\mathds{1}_{\{X_k>fX_{k-1},X_{k-1} \leq \tau \}} \xrightarrow{\P_1} 1,
\end{align*}
as $c_f \to \infty$. We first show ${\bf (i)}$. By Proposition \ref{prop:lower-bound-moments-U-1} (i), we have
\begin{align*}
\sum_{k = 1}^{n^{\ast}} \big|\E_{\P_1} U_k \big| \lesssim n^{\ast} f^{\a} \lesssim c_f^{-1}.
\end{align*}
Next, let $\mathcal{E} = 2\N \cap \{1,2,\ldots,n^{\ast}\}$ and  $\mathcal{O} = 2\N-1 \cap \{1,2,\ldots,n^{\ast}\}$ denote the even and odd subsets. Then by the triangle inequality,
independence and \eqref{eq:lower_bound:Vk:second_moment}, we have
\begin{align*}
\Big\|\sum_{k = 1}^{n^{\ast}} V_k \Big\|_2 \leq \Big\|\sum_{k \in \mathcal{E}} V_k \Big\|_2 + \Big\|\sum_{k \in \mathcal{O}} V_k \Big\|_2 \lesssim \sqrt{n^{\ast}} f^{\a/2} \lesssim c_f^{-1/2}.
\end{align*}
It remains to show {\bf (ii)}. This, however, is an immediate consequence of Lemma \ref{lem:lower-bound-product-of-Einsen} and Lemma \ref{lem:control:truncation}.
\end{proof}

\section{Acknowledgments}

We would like to thank Thomas Mikosch for his valuable comments regarding literature in extreme value theory.
This research was supported by the research grant ANR-19-CE40-0013.


\begin{thebibliography}{10}

\bibitem{aigner-lovell-schmidt1977}
D.~Aigner, C.~A.~K. Lovell, and P.~Schmidt.
\newblock Formulation and estimation of stochastic frontier production function
  models.
\newblock {\em Journal of Econometrics}, 6(1):21–37, 1977.

\bibitem{balakrishna:book:2021}
N.~Balakrishna.
\newblock {\em Non-{G}aussian autoregressive-type time series}.
\newblock Springer, Singapore, [2021] \copyright 2021.

\bibitem{barndorf:Jrssb:2001}
Ole~E. Barndorff-Nielsen and Neil Shephard.
\newblock Non-gaussian ornstein–uhlenbeck-based models and some of their uses
  in financial economics.
\newblock {\em Journal of the Royal Statistical Society: Series B (Statistical
  Methodology)}, 63(2):167--241, 2001.

\bibitem{bibinger-jirak-reiss2014}
M.~Bibinger, M.~Jirak, and M.~Reiss.
\newblock Volatility estimation under one-sided errors with applications to
  limit order books.
\newblock {\em Ann. Appl. Probab.}, 26(5):2754--2790, 2016.

\bibitem{BONDON:jmva:2009}
P.~Bondon.
\newblock Estimation of autoregressive models with epsilon-skew-normal
  innovations.
\newblock {\em Journal of Multivariate Analysis}, 100(8):1761--1776, 2009.

\bibitem{brockwell_davis_2016}
P.J. Brockwell and R.A. Davis.
\newblock {\em Introduction to time series and forecasting}.
\newblock Springer Texts in Statistics. Springer, [Cham], third edition, 2016.

\bibitem{dahlhaus2012}
R.~Dahlhaus.
\newblock Locally stationary processes.
\newblock In {\em Handbook of Statistics}, volume~30, page 351. Elsevier B.V.,
  2012.

\bibitem{dahlhaus_richter_wu_2019}
R.~Dahlhaus, S.~Richter, and W.B. Wu.
\newblock {Towards a general theory for nonlinear locally stationary
  processes}.
\newblock {\em Bernoulli}, 25(2):1013 -- 1044, 2019.

\bibitem{daouia_2021}
A.~Daouia, J.-P. Florens, and L.~Simar.
\newblock Robustified expected maximum production frontiers.
\newblock {\em Econometric Theory}, 37(2):346--387, 2021.

\bibitem{davis-knight-liu1992}
R.~A. Davis, K.~Knight, and J.~Liu.
\newblock M-estimation for autoregressions with infinite variance.
\newblock {\em Stochastic Processes and their Applications}, 40:145–180,
  1992.

\bibitem{davis-mccormick1989}
R.~A. Davis and W.~P. McCormick.
\newblock Estimation for first-order autoregressive processes with positive or
  bounded innovations.
\newblock {\em Stochastic Processes and their Applications}, 31(2):237–250,
  1989.

\bibitem{delleur1984}
J.~W. Delleur, V.~K. Gupta, R.~N. Bhattacharya, V.~Klemeš, R.~L. Smith,
  P.~Todorovic, S.~J. Deutsch, J.~A. Ramos, V.~T.~V. Nguyen, D.~K. Pickard,
  E.~M. Tory, J.~A. Smith, A.~F. Karr, and S.~W. Woolford.
\newblock Stochastic hydrology.
\newblock {\em Advances in Applied Probability}, 16(1):17–23, 1984.
\newblock DOI: 10.2307/1427220.

\bibitem{diaz-hughes-swetnam2010}
H.~F. Diaz, M.~K. Hughes, and T.~W. Swetnam.
\newblock {\em Dendroclimatology: progress and prospects}.
\newblock Springer Netherlands, 2010.

\bibitem{durbin}
J.~Durbin.
\newblock Efficient estimation of parameters in moving-average models.
\newblock {\em Biometrika}, 46(3/4):306--316, 1959.

\bibitem{eichler_2011}
M.~Eichler, G.~Motta, and R.~von Sachs.
\newblock Fitting dynamic factor models to non-stationary time series.
\newblock {\em J. Econometrics}, 163(1):51--70, 2011.

\bibitem{preve12030139}
Anders Eriksson, Daniel P.~A. Preve, and Jun Yu.
\newblock Forecasting realized volatility using a nonnegative semiparametric
  model.
\newblock {\em Journal of Risk and Financial Management}, 12(3), 2019.

\bibitem{farell1957}
M.~J. Farrell.
\newblock The measurement of productive efficiency.
\newblock {\em Journal of the Royal Statistical Society. Series A (General)},
  120(3):253–290, 1957.

\bibitem{feigin:1996:aaop}
Paul~D. Feigin, Marie~F. Kratz, and Sidney~I. Resnick.
\newblock {Parameter estimation for moving averages with positive innovations}.
\newblock {\em The Annals of Applied Probability}, 6(4):1157 -- 1190, 1996.

\bibitem{MR1288286}
Paul~D. Feigin and Sidney~I. Resnick.
\newblock Limit distributions for linear programming time series estimators.
\newblock {\em Stochastic Process. Appl.}, 51(1):135--165, 1994.

\bibitem{gaver-lewis1980}
D.~P. Gaver and P.~A.~W. Lewis.
\newblock First-order autoregressive gamma sequences and point processes.
\newblock {\em Advances in Applied Probability}, 12(3):727–745, 1980.

\bibitem{hall_2009}
P.~Hall and I.~Van Keilegom.
\newblock Nonparametric ``regression'' when errors are positioned at
  end-points.
\newblock {\em Bernoulli}, 15(3):614--633, 2009.

\bibitem{ing:jtsa:2018}
Wei-Cheng Hsiao, Hao-Yun Huang, and Ching-Kang Ing.
\newblock Interval estimation for a first-order positive autoregressive
  process.
\newblock {\em Journal of Time Series Analysis}, 39(3):447--467, 2018.

\bibitem{ing:jasa:positive:auto:2012}
Ching-Kang Ing and Chiao-Yi Yang.
\newblock Predictor selection for positive autoregressive processes.
\newblock {\em Journal of the American Statistical Association},
  109(505):243--253, 2014.

\bibitem{jirak-meister-reiss2014}
M.~Jirak, A.~Meister, and M.~Reiß.
\newblock Adaptive function estimation in nonparametric regression with
  one-sided errors.
\newblock {\em Annals of Statistics}, 42(5):1970–2002, 2014.
\newblock DOI: 10.1214/14-AOS1248.

\bibitem{knight2001}
K.~Knight.
\newblock Limiting distributions of linear programming estimators.
\newblock {\em Extremes}, 4(2):87--103, 2001.

\bibitem{kreiss_boot_2015}
J.-P. Kreiss and E.~Paparoditis.
\newblock Bootstrapping locally stationary processes.
\newblock {\em Journal of the Royal Statistical Society. Series B (Statistical
  Methodology)}, 77(1):267--290, 2015.

\bibitem{kreiss:aos:1987}
J.P. Kreiss.
\newblock {On Adaptive Estimation in Stationary ARMA Processes}.
\newblock {\em The Annals of Statistics}, 15(1):112 -- 133, 1987.

\bibitem{kumbhakar-park-simar-tsionas2007}
S.~C. Kumbhakar, B.~U. Park, L.~Simar, and E.~G. Tsionas.
\newblock Nonparametric stochastic frontiers: A local maximum likelihood
  approach.
\newblock {\em Journal of Econometrics}, 137(1):1–27, 2007.

\bibitem{lawrance-lewis1980}
A.~J. Lawrance and P.~A.~W. Lewis.
\newblock The exponential autoregressive-moving average {EARMA(p,q)} process.
\newblock {\em Journal of the Royal Statistical Society. Series B
  (Methodological)}, 42(2):150–161, 1980.

\bibitem{leadbetter1974}
M.~R. Leadbetter.
\newblock On extreme values in stationary sequences.
\newblock {\em Zeit\-schrift für Wahr\-schein\-lich\-keits\-theo\-rie und
  Ver\-wandte Ge\-bie\-te}, 28(4):289–303, 1974.
\newblock DOI: 10.1007/BF00532947.

\bibitem{MR691492}
M.~R. Leadbetter, Georg Lindgren, and Holger Rootz\'{e}n.
\newblock {\em Extremes and related properties of random sequences and
  processes}.
\newblock Springer Series in Statistics. Springer-Verlag, New York-Berlin,
  1983.

\bibitem{meeusen-vandenbroeck1977}
W.~Meeusen and J.~van~den Broeck.
\newblock Efficiency estimation from {C}obb-{D}ouglas production functions with
  composed error.
\newblock {\em International Eco\-no\-mic Review}, 18(2):435–444, 1977.

\bibitem{nielsen-shephard1999}
B.~Nielsen and N.~Shephard.
\newblock Likelihood analysis of a first-order autoregressive model with
  exponential innovations.
\newblock {\em Journal of time series analysis}, 24(3):337–344, 2003.

\bibitem{park-sickles-simar1998}
B.~U. Park, R.~C. Sickles, and L.~Simar.
\newblock Stochastic panel frontiers: {A} semiparametric approach.
\newblock {\em Journal of Econometrics}, 84:273–301, 1998.

\bibitem{park-sickles-simar2007}
B.~U. Park, R.~C. Sickles, and L.~Simar.
\newblock Semiparametric efficient estimation of dynamic panel data models.
\newblock {\em Journal of Econometrics}, 136(1):281–301, 2007.

\bibitem{PREVE2015S225}
D.l Preve.
\newblock Linear programming-based estimators in nonnegative autoregression.
\newblock {\em Journal of Banking \& Finance}, 61:S225--S234, 2015.
\newblock Recent Developments in Financial Econometrics and Applications.

\bibitem{richter_dahlhaus_2019}
S.~Richter and R.~Dahlhaus.
\newblock {Cross validation for locally stationary processes}.
\newblock {\em The Annals of Statistics}, 47(4):2145 -- 2173, 2019.

\bibitem{rosenblatt:book:2000}
M.~Rosenblatt.
\newblock {\em Gaussian and non-{G}aussian linear time series and random
  fields}.
\newblock Springer Series in Statistics. Springer-Verlag, New York, 2000.

\bibitem{selk2021multivariate}
L.~Selk, C.~Tillier, and O.~Marigliano.
\newblock Multivariate boundary regression models, 2021.

\bibitem{smith_1994}
R.~L. Smith.
\newblock Nonregular regression.
\newblock {\em Biometrika}, 81(1):173--183, 1994.

\bibitem{tsybakov2009}
A.~B. Tsybakov.
\newblock {\em Introduction to Nonparametric Estimation}.
\newblock Springer, 2009.

\bibitem{Vogt_aos_2012}
M.~Vogt.
\newblock Nonparametric regression for locally stationary time series.
\newblock {\em Ann. Statist.}, 40(5):2601--2633, 2012.

\bibitem{basawa:jtsa:2005}
J.~Zhou and I.~V. Basawa.
\newblock Maximum likelihood estimation for a first-order bifurcating
  autoregressive process with exponential errors.
\newblock {\em Journal of Time Series Analysis}, 26(6):825--842, 2005.

\end{thebibliography}

\newpage
\section{Supplement}\label{sec:supplement}

\subsection{Additional results for the upper bounds}

Throughout this section, we set $\P = \P_f$.

\begin{proof}[Proof of Lemma \ref{lem:Norm-X2}]
Since $f((k-l)/n)\le\rho$ results in $f_{k,i}=\prod_{l=0}^{i-1}f\left(\frac{k-l}n\right)\le\rho^i$, we thus deduce
\begin{align*}
\left\|X_k^{(2)}\right\|_1
&\le\sum_{i>n^\gamma}f_{k,i}\left\|\varepsilon_{k-i}\right\|_1
\lesssim\sum_{i>n^\gamma}\rho^i=\frac{\rho^{n^\gamma+1}}{1-\rho}
\lesssim\rho^{\left(n^\gamma\right)},
\end{align*}
and similarly $\|X_k\|_p \lesssim 1$.
\end{proof}

\begin{proof}[Proof of Lemma \ref{lem:Leadbetters-D-c}]
Due to $C_{\eta_n}\subseteq C_{\eta_n}^{(1)}$,
\begin{eqnarray*}
\left|\P\left(C_{\eta_n,n}^{(1)}\right)-\P\left(C_{\eta_n,n}\right)\right|&=&\P\left(\bigcap_{k=1}^n\left\{\varepsilon_k\ge uX_{k-1}^{(1)}\right\}\backslash\bigcap_{k=1}^n\left\{\varepsilon_k\ge uX_{k-1}\right\}\cap A_{\eta_n}\right).
\end{eqnarray*}
Using
\begin{eqnarray*}
\bigcap_{k=1}^n\left\{\varepsilon_k\ge uX_{k-1}^{(1)}\right\}\backslash\bigcap_{k=1}^n\left\{\varepsilon_k
\ge uX_{k-1}\right\}
&\!\subseteq\!&\bigcup_{k=1}^n\left\{\varepsilon_k\ge uX_{k-1}^{(1)}\right\}\backslash\left\{\varepsilon_k
\ge uX_{k-1}\right\},
\end{eqnarray*}
we can conclude
\begin{align}
\left|\P\left(C_{\eta_n,n}^{(1)}\right)-\P\left(C_{\eta_n,n}\right)\right|
&\le\sum_{k=1}^n\P\left(\left\{\varepsilon_k\ge uX_{k-1}^{(1)}\right\}\backslash\left\{\varepsilon_k\ge
uX_{k-1}\right\}\cap A_{\eta_n}\right).\label{eq:sum-for-each-summand}
\end{align}
For each summand holds
\begin{eqnarray}
&&\P\left(\left\{\varepsilon_k\ge uX_{k-1}^{(1)}\right\}\backslash\left\{\varepsilon_k\ge uX_{k-1}\right\}
\cap A_{\eta_n}\right)\nonumber\\
&=&\P\left(\left\{\u X_{k-1}^{(1)}\le\varepsilon_k<\u X_{k-1}^{(1)}+\u X_{k-1}^{(2)}\right\}\cap
A_{\eta_n}\right)\label{eq:for-each-summand-for-later}\\
&\le&\P\left(\u X_{k-1}^{(1)}\le\varepsilon_k<\u X_{k-1}^{(1)}+\u\eta_n\right)\nonumber,
\end{eqnarray}
which can be written as
\begin{eqnarray*}
&&\int_0^\infty\P\left(\u x\le\varepsilon_k<\u x+\u\eta_n|X_{k-1}^{(1)}=x\right)dF_{X_{k-1}^{(1)}}
(x)\\
&=&\int_0^\infty\P\left(\u x\le\varepsilon_k<\u x+\u\eta_n\right)dF_{X_{k-1}^{(1)}}(x)\\
&=&\int_0^\infty F_\varepsilon\left(\u x+\u\eta_n\right)-F_\varepsilon\left(\u x\right)dF_{X_{k-1}^{(1)}}(x),
\end{eqnarray*}
where we have used that $X_{k-1}^{(1)}=\sum_{i=0}^{n^\gamma}f_{k-1,i}\;\varepsilon_{k-1-i}$ is independent of $\varepsilon_k$.
Using that $F_{\varepsilon} \in \mathcal{H}(L_{\varepsilon},\beta_{\varepsilon})$ by Assumption \ref{ass:main} {\bf (i)}, this is further bounded
by $L_{\varepsilon} (u \eta_n)^{\beta_{\varepsilon}}$.

For each summand in \eqref{eq:sum-for-each-summand}, we thus have
\begin{equation*}
\P\left(\left\{\varepsilon_k\ge uX_{k-1}^{(1)}\right\}\backslash\left\{\varepsilon_k\ge uX_{k-1}\right\}\cap A_{\eta_n}\right)\lesssim (u \eta_n)^{\beta_{\varepsilon}}.
\end{equation*}
This finally yields
\begin{align*}
\left|\P\left(C_{\eta_n}\right)-\P\left(C_{\eta_n}^{(1)}\right)\right|\lesssim n (u \eta_n)^{\beta_{\varepsilon}}.
\end{align*}
\end{proof}


\begin{proof}[Proof of Lemma \ref{lem:neighboring}]
We first note that for all $j\ge1$,
\begin{eqnarray*}
&&\P\left(\varepsilon_k\le\u X_{k-1},\varepsilon_{k+j}\le\u X_{k+j-1}\right)\\
&=&\P\Big(\varepsilon_k\le\u X_{k-1},\varepsilon_{k+j}\le\u\sum_{i=1}^jf_{k+j-1,i-1}\;\varepsilon_{k+j-i}+\u f_{k+j-1,j}\;X_{k-1}\Big)\\
&=&\P\Big(\varepsilon_k\le\u X_{k-1},\varepsilon_{k+j}\le\u\sum_{i=1}^j\tilde f_i\;\varepsilon_{k+j-i}+\u\tilde f_{j+1}X_{k-1}\Big),
\end{eqnarray*}
where $\tilde f_i:=f_{k+j-1,i-1}$ with $f_{k,i}$ defined as in \eqref{eq:f_(k,i)-Def}. We can further write this as
\begin{eqnarray}
&&\P\Big(\varepsilon_k\le\u X_{k-1},\varepsilon_{k+j}\le\u\sum_{i=1}^{j-1}\tilde f_i\;\varepsilon_{k+j-i}+\u\tilde f_j\;\varepsilon_k+\u\tilde f_{j+1}X_{k-1}\Big)\nonumber\\
&\le&\P\Big(\varepsilon_k\le\u X_{k-1},\varepsilon_{k+j}\le\u X_{k+j-1}'+\u^2\tilde f_jX_{k-1}+\u\tilde f_{j+1}X_{k-1}\Big),\label{eq:indep-copy}
\end{eqnarray}
where $X_{k+j-1}'$ is an independent copy of $X_{k+j-1}$ in the sense that $X_{k+j-1}'\overset d=X_{k+j-1}$ and $X_{k+j-1}'$ is independent of $\{\varepsilon_i\}_{i\le k}$ and $\{\varepsilon_i\}_{i\ge k+j}$. Using that $X_{k-1}$ is independent of $\varepsilon_k$, $\varepsilon_{k+j}$ and $X_{k+j-1}'$, \eqref{eq:indep-copy} can, analogously to the proof of Proposition \ref{thm:err-distrib}, be written as
\begin{align}
&\int_0^\infty\P\Big(\varepsilon_k\le\u x\Big)\P\Big(\varepsilon_{k+j}\le\u X_{k+j-1}'+\u^2\tilde f_jx+\u\tilde f_{j+1}\;x\Big)dF_{X_{k-1}}(x)\nonumber\\
=&\int_0^{x_n}\P\Big(\varepsilon_k\le\u x\Big)\P\Big(\varepsilon_{k+j}\le\u X_{k+j-1}'+\u^2\tilde f_jx+\u\tilde f_{j+1}\;x\Big)dF_{X_{k-1}}(x)+\nonumber\\
&\int_{x_n}^\infty\P\Big(\varepsilon_k\le\u x\Big)\P\Big(\varepsilon_{k+j}\le\u X_{k+j-1}'+\u^2\tilde f_jx+\u\tilde f_{j+1}\;x\Big)dF_{X_{k-1}}(x)\nonumber\\
=:&I_n+II_n\label{eq:Doppelintegral}
\end{align}
for $x_n=u^{\xi-1}$, $\xi\in(0,1)$, with $x_n\to\infty$ and $\u x_n\to0$ as $n\to\infty$. First, we deal with $I_n$. Since $X_{k+j-1}'$ is independent of $\varepsilon_{k+j}$, we obtain
\begin{eqnarray*}
&&\P\Big(\varepsilon_{k+j}\le\u X_{k+j-1}'+\u^2\tilde f_jx+\u\tilde f_{j+1}\;x\Big)\\
&=&\int_0^\infty\P\Big(\varepsilon_{k+j}\le\u z+\u^2\tilde f_jx+\u\tilde f_{j+1}\;x\Big)dF_{X_{k+j-1}}(z).
\end{eqnarray*}
With $z_n=\u^{\zeta-1}$, $\zeta\in(0,1)$, such that $z_n\to\infty$ and $\u z_n\downarrow0$ as $n\to\infty$, for $x\in(0,x_n)$, this is bounded by
\begin{eqnarray}
&&\int_0^{z_n}\c\left(\u z+\u^2\tilde f_jx+\u\tilde f_{j+1}\;x\right)^\a+\nonumber\\
&&\hspace{20mm}\mathcal{O}\left(\left(\u z+\u^2\tilde f_jx+\u\tilde f_{j+1}\;x\right)^{\a+\delta}\right)dF_{X_{k+j-1}}(z)+\nonumber\\
&&\int_{z_n}^\infty\P\Big(\varepsilon_{k+j}\le\u z+\u^2\tilde f_jx+\u\tilde f_{j+1}\;x\Big)dF_{X_{k+j-1}}(z)\nonumber\\
&=:&\tilde I_n+\tilde{II}_n,\label{eq:lemma-for-Condition-D'-tildeI(I)n}
\end{eqnarray}
where we also used Assumption \ref{ass:main}. For $\tilde I_n$, we have
\begin{align*}
\tilde I_n\le&\c\u^\a\E\left[\left(X_{k+j-1}'+\u\tilde f_jx+\tilde f_{j+1}\;x\right)^\a\right]+\\
&\mathcal{O}\left(\u^{\a+\delta}\right)\E\left[\left(X_{k+j-1}'+\u\tilde f_jx+\tilde f_{j+1}\;x\right)^{\a+\delta}\right].
\end{align*}
Concerning $\tilde{II}_n$, by Markov's inequality and Assumption \ref{ass:main} {\bf (ii)}, there exists $\tilde q(1-\zeta)\ge\a+\tilde\delta$, $p>\tilde{q}$, $\tilde{\delta} > 0$, such that
$$\tilde{II}_n\le\int_{z_n}^\infty dF_{X_{k+j-1}}(z)=\P(X_{k+j-1}'\ge z_n)\le\left\|X_{k+j-1}'\right\|_{\tilde q}^{\tilde q}z_n^{-\tilde q}=\mathcal{O}\left(\u^{\a+\tilde\delta}\right),$$
selecting $\zeta > 0$ sufficiently small. Hence, the sum in \eqref{eq:lemma-for-Condition-D'-tildeI(I)n} is bounded by
\begin{eqnarray*}
\tilde I_n+\tilde{II}_n&\le&\c\u^\a\E\left[\left(X_{k+j-1}'+\u\tilde f_jx+\tilde f_{j+1}\;x\right)^\a\right]+\\
&&\mathcal{O}\left(\u^{\a+\delta}\right)\E\left[\left(X_{k+j-1}'+\u\tilde f_jx+\tilde f_{j+1}\;x\right)^{\a+\delta}\right]+\mathcal{O}\left(\u^{\a+\tilde\delta}\right).
\end{eqnarray*}
For $I_n$ from \eqref{eq:Doppelintegral}, we thus obtain by Cauchy-Schwarz
\begin{eqnarray*}
I_n&\le&\int_0^{x_n}\P\Big(\varepsilon_k\le\u x\Big)\Bigg(\c\u^\a\E\left[\left(X_{k+j-1}'+\u\tilde f_jx+\tilde f_{j+1}\;x\right)^\a\right]+\\
&&\mathcal{O}\left(\u^{\a+\delta}\right)\E\left[\left(X_{k+j-1}'+\u\tilde f_jx+\tilde f_{j+1}\;x\right)^{\a+\delta}\right]+\mathcal{O}\left(\u^{\a+\tilde\delta}\right)\Bigg)dF_{X_{k-1}}(x)\\
&\le&\Bigg(\c\u^\a\E^{\frac{1}{2}}\left[\left(X_{k+j-1}'+\u\tilde f_jX_{k-1}+\tilde f_{j+1}\;X_{k-1}\right)^{2\a}\right]+\\
&&\mathcal{O}\left(\u^{\a+\delta}\right)\E^{\frac{1}{2}}\left[\left(X_{k+j-1}'+\u\tilde f_jX_{k-1}+\tilde f_{j+1}\;X_{k-1}\right)^{2\a+2\delta}\right]+\mathcal{O}\left(\u^{\a+\tilde\delta}\right)\Bigg)\\
&&\Big(\int_0^{x_n}\P^2\big(\varepsilon_k\le\u x\big)dF_{X_{k-1}}(x) \Big)^{\frac{1}{2}} \\&\lesssim& \Big(u^{\a} + \mathcal{O}\big(u^{\a + \tilde{\delta}}\big) \Big) \Big(\int_0^{x_n}\P^2\big(\varepsilon_k\le\u x\big)dF_{X_{k-1}}(x) \Big)^{\frac{1}{2}}.
\end{eqnarray*}
Now, using Assumption \ref{ass:main}, we have (with $\tilde{\delta}>0$ sufficiently small such that $p \geq 2\a + 2\tilde{\delta}$)
\begin{align*}
\int_0^{x_n}\P^2\big(\varepsilon_k\le\u x\big)dF_{X_{k-1}}(x) \lesssim u^{2\a} \E X_{k-1}^{2\a} + u^{2\a + 2\tilde{\delta}} \E X_{k-1}^{2\a + 2 \tilde{\delta}},
\end{align*}
and hence $I_n \lesssim u^{2\a}$. For $II_n$ in \eqref{eq:Doppelintegral}, selecting $\xi>0$ sufficiently small, Markov's inequality yields
\begin{align}
II_n\le&\int_{x_n}^\infty\,dF_{X_{k-1}}(x)
=\P\left(X_{k-1}\ge x_n\right)
\le\left\|X_{k-1}\right\|_p^p x_n^{-p}
\lesssim\u^{2\a},\label{eq:lemma-for-Condition-D'-II_n}
\end{align}
using Assumption \ref{ass:main} {\bf (ii)}. This completes the proof.
\end{proof}

\subsection{Additional results for the lower bounds}

\begin{proof}[Proof of Lemma \ref{lem:lower-bound-moments-U-2}]
As will be obvious from the proof, we may set $\tau = \infty$, hence $\mathds{1}_{\{X_{k-1} \leq \tau\}} = 1$. Moreover, since the case $\a = 1$ is trivial, we will assume $\a > 1$ below. Recall that $X_k,X_{k-1}$ are independent subject to $\P_1$. Set $X = X_{k-1}$. Then
\begin{align}
&\E_{\P_1}\frac{X_{k-1}}{X_k}\mathds{1}_{\{X_k>fX_{k-1}\}}\nonumber\\
=&\E_{\P_1}\left[X\frac{\b^\a}{\Gamma(\a)}\int_0^\infty t^{\a-2}e^{-\b t}\d t\right]-
\E_{\P_1}\left[X\frac{\b^\a}{\Gamma(\a)}\int_0^{fX}t^{\a-2}e^{-\b t}\d t\right].\label{eq:lower-bound-lemma-for-Erwartungswert}
\end{align}
Since $\Gamma(\a)=(\a-1)\Gamma(\a-1)$ for $\a>1$, we obtain for the first summand
\begin{align}
(\a-1)f\E_{\P_1}\left[X\frac{\b^\a}{\Gamma(\a)}\int_0^\infty t^{\a-2}e^{-\b t}\d t\right]
=&\b f\E_{\P_1}\left[X\int_0^\infty f_{\Gamma(\a-1,\b)}(t)\d t\right]\nonumber\\
=&\b f\E_{\P_1}\left[X_{k-1}\right],\label{eq:lower-bound-lemma-for-Erwartungswert-1}
\end{align}
where $f_{\Gamma(\a-1,\b)}$ denotes the density of a $\Gamma(\a-1,\b)$-distributed random variable. For the second summand, we have
\begin{align}
(\a-1)f\E_{\P_1}\left[X\frac{\b^\a}{\Gamma(\a)}\int_0^{fX}t^{\a-2}e^{-\b t}\d t\right]
\lesssim&f\E_{\P_1}\left[X\int_0^{fX_{}}t^{\a-2}\d t\right]\nonumber\\
=&f^{\a}\frac{\E_{\P_1} X_{}^\a}{\a-1}
\lesssim f^\a.\label{eq:lower-bound-lemma-for-Erwartungswert-2}
\end{align}
Combining \eqref{eq:lower-bound-lemma-for-Erwartungswert-1} and \eqref{eq:lower-bound-lemma-for-Erwartungswert-2} with \eqref{eq:lower-bound-lemma-for-Erwartungswert} gives the result.
\end{proof}

\begin{proof}[Proof of Lemma \ref{lem:lower-bound-moments-U-1}]
As in the proof of Lemma \ref{lem:lower-bound-moments-U-2} we use an independent copy $X$ of $X_{k-1}$, yielding
\begin{eqnarray}
& \left\|f\frac{X_{k-1}}{X_k}\mathds{1}_{\{X_k>fX_{k-1}\}} \mathds{1}_{\{X_{k-1} \leq \tau\}}\right\|_p^p \leq \E_{\P_1}\left(fX\right)^p\int_{fX}^1t^{\a-1-p}e^{-bt}\d t\nonumber\\
&+\E_{\P_1}\left(fX\right)^p\int_1^\infty t^{\a-1-p}e^{-bt}\d t \mathds{1}_{\{X \leq \tau\}}.\label{eq:lower-bound-moments-U}
\end{eqnarray}
For the first integral in \eqref{eq:lower-bound-moments-U}, we get
\begin{align*}
\int_{fX}^1t^{\a-1-p}e^{-bt}\d t \le&\frac{\left(fX\right)^{\a-p}-1}{p-\a}.
\end{align*}

Since $\E_{\P_1}X^\a <\infty$, we thus obtain for the first expectation in \eqref{eq:lower-bound-moments-U}
\begin{align}
\E_{\P_1} \left(fX\right)^p\int_{fX}^1t^{\a-1-p}e^{-bt}\d t
&\le\E_{\P_1} \left(fX\right)^p\frac{\left(fX\right)^{\a-p}-1}{p-\a}
\lesssim\frac{f^\a}{p-\a}.\label{eq:lower-bound-lemma-for-exp-u-1}
\end{align}
We now turn to the second expectation in \eqref{eq:lower-bound-moments-U}. Since
\begin{align*}
\int_1^\infty t^{\a-1-p}e^{-bt}\d t&\le \Gamma(\mathfrak{a}) \int_1^\infty f_\varepsilon(t)\d t\le \Gamma(\mathfrak{a}),
\end{align*}
we obtain
\begin{align}
\E_{\P_1}\left(fX\right)^p \mathds{1}_{\{X \leq \tau\}}\int_1^\infty t^{\a-1-p}e^{-bt}\d t &\lesssim f^p \tau^{(p-\mathfrak{a})\vee 0} \E X^{\mathfrak{a}} \lesssim f^p \tau^{(p-\mathfrak{a})\vee 0}.\label{eq:lower-bound-lemma-for-exp-u-2}
\end{align}
Combining \eqref{eq:lower-bound-lemma-for-exp-u-1} and \eqref{eq:lower-bound-lemma-for-exp-u-2} with \eqref{eq:lower-bound-moments-U}, we obtain
\begin{align*}
\E_{\P_1}\left(f\frac{X_{k-1}}{X_k}\right)^p\mathds{1}_{\{X_k>fX_{k-1}\}} \mathds{1}_{\{X_{k-1} \leq \tau\}} \lesssim\frac{f^\a}{p-\a}+f^p \tau^{(p-\mathfrak{a})\vee 0}.
\end{align*}
\end{proof}

\begin{proof}[Proof of Lemma \ref{lem:lower-bound-product-of-Einsen}]
If a sequence $X_n$ of random variables only takes values in the set $\{0,1\}$, it suffices to show that
$\E X_n \to 0$ to establish $X_n \xrightarrow{\P} 0$. To this end, we assume w.l.o.g. that $n^{\ast}$ is even, and consider $\mathcal{I},\mathcal{J} \subset \{1,\ldots,n^{\ast}\}$ such that
\begin{align*}
\mathcal{I} = \{2,6,10,\ldots\}, \quad \mathcal{J} = \{4,8,12,\ldots\}.
\end{align*}
Then, by independence
\begin{align*}
\E_{\P_1} \prod_{k \in \mathcal{I}} \mathds{1}_{\{X_k > f X_{k-1}\}} = \prod_{k \in \mathcal{I}} \E_{\P_1} \mathds{1}_{\{X_k > f X_{k-1}\}}.
\end{align*}
 Next, again by independence, we have
\begin{align*}
\Gamma(\a)\E_{\P_1} \mathds{1}_{\{X_k \leq f X_{k-1}\}} \leq \E_{\P_1} \int_0^{f X_{k-1}} x^{\a - 1} dx \lesssim f^{\a}.
\end{align*}
Together with the above, this implies
\begin{align*}
\E_{\P_1} \prod_{k \in \mathcal{I}} \mathds{1}_{\{X_k > f X_{k-1}\}}\geq \big(1 - O(f^{\a})\big)^{n^{\ast}/2} = e^{O(c_f^{-1})} = 1 + o\big(1\big),
\end{align*}
as $c_f \to \infty$, and hence
\begin{align}
\prod_{k \in \mathcal{I}} \mathds{1}_{\{X_k > f X_{k-1}\}} \xrightarrow{\P_1} 1.
\end{align}
We may argue analogously for $\mathcal{J}$. Since clearly $\mathds{1}_{\{X_1 > f X_{0}\}} \xrightarrow{\P_1} 1$, the claim follows.
\end{proof}

\begin{proof}[Proof of Lemma \ref{lem:control:truncation}]
Since by Markov's inequality and $t > 0$
\begin{align*}
\P\big(\max_{0 \leq k \leq N} X_k \geq \tau_N^2 \big) \leq e^{-t \tau_N^2} (N+1) \E e^{tX_0},
\end{align*}
this follows from the fact that for $t < \b$, we have
\begin{align*}
M_{X_0}\big(t\big) = \E e^{tX_0} = \big(1 - t/\b\big)^{-\a}.
\end{align*}
\end{proof}

\end{document}